\documentclass[preprint,12pt]{elsarticle}

\usepackage{amssymb}
\usepackage{amsmath}
\usepackage{color}
\usepackage{amsthm}
\allowdisplaybreaks[4]
\newtheorem{lm}{{\bf {Lemma}}}[section]
\newtheorem{thm}{{\bf {Theorem}}}[section]
\newtheorem{cor}{{\bf {Corollary}}}[section]
\newtheorem{Rm}{{\bf {Remark}}}[section]

\begin{document}
\begin{frontmatter}

\title{Inhomogeneous Picard-Fuchs equations of Abelian integrals in piecewise smooth near-Hamiltonian systems}

\author{Hefei Zhao}
\ead{mathzhaohefei@126.com}
\author{Yun Tian\corref{cor1}}
\cortext[cor1]{Corresponding author.}
\ead{ytian22@shnu.edu.cn}
\address{Department of Mathematics, Shanghai Normal University, Shanghai, China}

\begin{abstract}
In this paper, we explicitly obtain inhomogeneous Picard-Fuchs equations for Abelian integrals $I_{i,j}^+(h)$ , 
where $I_{i,j}^+(h)$ is an integral along orbital arcs defined by polynomials $\frac{1}{2}y^2 + F(x)=h$. 
Moreover, we discuss the method of using Picard-Fuchs equations to  recursively compute  
the asymptotic expansions of genearating functions of Abelian integrals near a homoclinic loop. 
As an application, we derive the maximum number of isolated zeros of Melnikov functions 
near a nilpotent saddle homoclinic loop for piecewise polynomials perturbations with the inclination $\theta$ of
the separation line as a free parameter. 
\end{abstract}

\end{frontmatter}

\section{Introduction and main results }

It is well known that the algebraic structure of Abelian integrals and the associated Picard-Fuchs equations 
play an important role in the analysis of limit cycle bifurcations in (smooth) polynomial near-Hamiltonian systems.
The algebraic structure of Abelian integrals provides an efficient way to simplify Abelian integrals 
as linear combinations of a finite number of generating functions. 
By the Picard-Fuchs equations of generating functions we can study the qualitative property of Abelian integrals
to estimate the number of zeros of some Melnikov functions, see for example \cite{CCH1985,NM2017,NY2001,ZZ1999} and references therein.

Recently, the method of Picard-Fuchs equations has been applied in the study of limit cycles
 for some piecewise smooth near-Hamiltonian systems. For example see \cite{S2024,YZ2018},
where homogeneous Picard-Fuchs equations are obtained for a separation line fixed.
Then a question arises naturally: How far can the method of Picard-Fuchs equations be extended
in piecewise smooth near-Hamiltonian systems?

\subsection{Inhomogeneous Picard-Fuchs equations}
In this paper, we shall firstly derive Picard-Fuchs equations for
 piecewise smooth polynomial perturbations of planar Hamiltonian systems 
with Hamiltonians $H(x,y)=\frac{1}{2}y^2 + F(x)$, where $F(x)$ is a polynomial of degree $n$.
Take the straight line $L(x,y,\theta)=0$ as the separation line, where
\begin{equation*}\label{sp1}
L(x,y,\theta)= \sin(\theta)(x-c)-\cos(\theta) y,\quad 0<\theta\leq\pi,
\end{equation*}
and $c$ is a constant. Then the perturbed Hamiltonian system can be written as
\begin{equation}\label{e1}
\begin{split}
(\dot{x},\dot{y})=\left\{
\begin{aligned}
(H_y(x,y)+\varepsilon P^+(x,y),\,\,&-H_x(x,y)+\varepsilon Q^+(x,y)) ,\quad (x,y)\in\Sigma^+,\\[1ex]
(H_y(x,y)+\varepsilon P^-(x,y),\,\,&-H_x(x,y)+\varepsilon Q^-(x,y)) ,\quad (x,y)\in\Sigma^-,
\end{aligned}
\right.
\end{split}
\end{equation}
where  $|\varepsilon|\ll1$, $\Sigma^\pm=\{(x,y)|\pm L(x,y,\theta)>0\}$,
$P^\pm(x,y)$ and $Q^\pm(x,y)$ are polynomials in $(x,y)$.
Suppose that the unperturbed system \eqref{e1}$|_{\varepsilon=0}$ has a continuous period annulus
$\mathcal{A}=\{\Gamma_h| h\in(\alpha,\beta)\}$, where $\Gamma_h\subset H^{-1}(h)$ is a periodic orbit.
Further, we suppose that the separation line of \eqref{e1} splits $\Gamma_h$ into two nonempty connected parts 
$\Gamma_h^\pm=\Gamma_h\cap \Sigma^\pm$.

By \cite{LiuHan2010} the number of limit cycles produced from $\mathcal{A}$ in system \eqref{e1}
can be estimated by the number of isolated zeros of Melnikov function 
\begin{equation}\label{e2}
\widetilde M(h)=\,\int_{\Gamma_h^-}Q^-(x,y)\mathrm{d}x-P^-(x,y)\mathrm{d}y+\int_{\Gamma_h^+}Q^+(x,y)\mathrm{d}x-P^+(x,y)\mathrm{d}y
\end{equation}
for $h\in(\alpha,\beta)$. 
As we can see from \eqref{e2} that $\widetilde M(h)$ is a linear combination of Abelian integrals along $\Gamma_h^\pm$.
Before we present the Picard-Fuchs equations for these integrals, we should give a group of generating functions for $\widetilde M(h)$.

\begin{thm}\label{mth1}
Let 
\begin{equation}\label{h4a}
I_{i,j}(h)=\oint_{\Gamma_h}x^iy^j\mathrm{d}x,\quad
I_{i,j}^{+}(h)=\int_{\Gamma_h^+}x^iy^j\mathrm{d}x,\quad K_{i,j+1}(h)= \int_{\Gamma_h^+}\mathrm{d}(x^iy^{j+1}).
\end{equation}
For system \eqref{e1} the Melnikov function $\widetilde{M}(h)$ in \eqref{e2} can be written as
\begin{equation}\label{h4b}
\widetilde{M}(h)=\sum_{i=0}^{n-2}p_i(h)I_{i,1}(h)+\sum_{i=0}^{n-2}q_i(h)I^+_{i,1}(h)
+\sum_{i=0}^{n-2}s_{i}(h)\widetilde{K}_{i}(h),
\end{equation}
where $p_i(h)$, $q_i(h)$ and $s_{i}(h)$, $i=0,1,\ldots,n-2$, are polynomials in $h$, and
$\widetilde K_i(h)=K_{i+1,0}(h)$ for the separation line $y=0$,
otherwise $\widetilde K_i(h)=K_{0,i+1}(h)$.
\end{thm}

Theorem \ref{mth1} shows that for system \eqref{e1} any Abelian integral along orbital arcs $\Gamma_h^\pm$ 
can be expressed as a linear combination of generating functions $I_{i,1}(h)$, $I_{i,1}^+(h)$ and $\widetilde K_i(h)$,
$i=0,1,\dots,n-2$.
By \cite{Petrov1984} we know that Abelian integrals along orbits $\Gamma_h$  can be expressed 
as  linear combinations of $I_{i,1}(h)$, $0\le i\le n-2$.
To prove Theorem \ref{mth1}, 
in Lemma \ref{st1} we prove that any Abelian integral along $\Gamma_h^+$ can be expressed
as a linear combination of integrals $I_{i,1}^+(h)$ and $K_{i,j}(h)$.
Furthermore, the algebraic structure of functions $K_{i,j}(h)$ is also given in Section \ref{sec2}. 

Picard-Fuchs equations of $\mbox{\boldmath $X$}(h)={\mbox{\rm col}}(I_{0,1}(h),\,I_{1,1}(h),\,\ldots,\,I_{n-2,1}(h))$ have been extensively applied 
in the study of bifurcations of limit cycles for smooth polynomial perturbations.
For more information see the book \cite{CL2007} and references therein.
For the sake of convenience, the explicit expression of Picard-Fichs equations of $\mbox{\boldmath $X$}(h)$ is presented in Section \ref{sec2}.

For $\widetilde{\mbox{\boldmath $X$}}(h)={\mbox{\rm col}}(I^+_{0,1}(h),\,I^+_{1,1}(h),\,\ldots,\,I^+_{n-2,1}(h))$, 
the associated Picard-Fuchs equations are given in the following theorem.

\begin{thm} \label{st2}
Let $H(x,y)=\frac{1}{2}y^2+b_1x + b_2x^2 + \cdots + b_n x^n$, where $b_n\neq0$, $n\geq2$.
Suppose that for any $h\in(\alpha,\beta)$ the orbital arc $\Gamma_h^+$ of the unperturbed system \eqref{e1}$|_{\varepsilon=0}$ 
is from the point $(x_1(h),y_1(h))$ to the point $(x_2(h),y_2(h))$. 
 Then the column vector function $\widetilde{\mbox{\boldmath $X$}}(h)$ satisfies the Picard-Fuchs equations
\begin{equation}\label{t5}
2(\mbox{\boldmath $T$}_1+\mbox{\boldmath $T$}_2\mbox{\boldmath $T$}_4^{-1}\mbox{\boldmath $T$}_3)(\widetilde{\mbox{\boldmath $X$}}'(h)-
\mbox{\boldmath $J$}(h)) + 2\mbox{\boldmath $T$}_2\mbox{\boldmath $T$}_4^{-1}\mbox{\boldmath $K$}(h)=\widetilde{\mbox{\boldmath $X$}}(h),
\end{equation}
where $\mbox{\boldmath $T$}_1$, $\mbox{\boldmath $T$}_2$, $\mbox{\boldmath $T$}_3$ and $\mbox{\boldmath $T$}_4$ are 
matrices given in \eqref{st3}, and
\begin{equation}\label{e4}
\begin{split}
\mbox{\boldmath $K$}(h) =&\, {\mbox{\rm col}}(K_{0,1}(h),\,K_{1,1}(h),\,\ldots,\, K_{n-1,1}(h)),\\
\mbox{\boldmath $J$}(h) =&\, {\mbox{\rm col}}(J_0(h),\,J_1(h),\,\ldots,\,J_{n-2}(h)),\\
\end{split}
\end{equation}
with
$J_i(h) = x_2^i(h) y_2(h)x'_2(h) - x_1^i(h) y_1(h)x'_1(h)$.
\end{thm}

\begin{Rm}
In Theorem \ref{st2}, it is assumed that $\Gamma_h^+$ is an orbital arc contained by the periodic orbit $\Gamma_h$ for system \eqref{e1}.
From its proof, we can see that Picard-Fuchs equations \eqref{t5} still hold 
for any continuous family of orbital arcs of the unperturbed system  of \eqref{e1}.  
\end{Rm}



\subsection{Asymptotic expansion of Melnikov functions}
Next, we will discuss how to apply Picard-Fuchs equations \eqref{t5} 
to study bifurcations of limit cycles in system \eqref{e1}.
For Poincar\'e bifurcation of the period annulus $\mathcal{A}$, if we do not have
the explicit expressions of $x_i(h)$ and $y_i(h)$ from 
\begin{equation}\label{e5}
L(x_i,y_i,\theta)=0, \quad H(x_i,y_i)=h,\quad h\in(\alpha,\beta),
\end{equation}
it could be very difficult to study the properties of Abelian integrals $I_{i,1}^+(h)$ 
on the interval $(\alpha,\beta)$ by equations \eqref{t5}.
Note that if the $x$-axis is the separation line, then $y_1(h)=y_2(h)\equiv0$, which implies that 
$\mbox{\boldmath $K$}(h) = \mbox{\boldmath $J$}(h) \equiv 0$ by \eqref{e4}.
Then Picard-Fuchs equations \eqref{t5} become homogeneous,
and can be used for Poincar\'e bifurcation just like for smooth polynomial perturbations,
 because  $2\widetilde{\mbox{\boldmath $X$}}(h)=\mbox{\boldmath $X$}(h)$ in this case.

It is more feasible to derive the asymptotic expansions of $x_i(h)$ and $y_i(h)$ at $h=h_0$ from \eqref{e5}, where $\alpha\le h_0\le \beta$.
Then we can compute the asymptotic expansion of $\widetilde{\mbox{\boldmath $X$}}(h)$ at $h_0$ by Picard-Fuchs equations \eqref{t5},
which makes it possible for us to derive the asymptotic expansion of $\widetilde M(h)$.
Therefore, we can study isolated zeros of $\widetilde M(h)$ near $h=h_0$ to determine limit cycles bifurcating near $\Gamma_{h_0}$ in system \eqref{e1},
where $\Gamma_{h_0}$ could be a center or a homoclinic/heteroclinic loop.

For homoclinic or heteroclinic bifurcation of limit cycles in piecewise smooth near-Hamiltonian systems, 
there are some papers about the compuation of the corresponding asymptotic expansions of Melnikov functions
with a fixed separation line (see \cite{LHR2012,LHZ2013,LiuHan2023,WeiZhang2016,XH2021} for example),
where formulas are given only for the first few coefficients.

Assume that $\Gamma_{\beta}$ represents a homoclinic loop for $h_0=\beta$.
We can apply Theorems \ref{mth1} and \ref{st2} 
to recursively compute the coefficients of the asymptotic expansion of  $\widetilde M(h)$ near $h=\beta$. 
The method is illustrated in Section \ref{sec3}.
It is worthy to mention that this method can be extented to the cases 
where $\Gamma_{h_0}$ is a center or a heteroclinic loop of \eqref{e1}$|_{\varepsilon=0}$.

As an application, to show the effects of piecewise smooth perturbations on homoclinic bifurcation of limit cycles, 
we consider smooth and piecewise smooth polynomial perturbations for the quartic Hamiltonian system
\begin{equation}\label{hsys}
\dot x= y,\quad \dot y = x^3(1-x),
\end{equation}
respectively. The Hamiltonian $H(x,y)$ becomes into 
$$H(x,y)=\frac{1}{2} y^2 - \frac{1}{4} x^4 + \frac{1}{5} x^5.$$
System \eqref{hsys} has an elementary center $(1,0)$ and a nilpotent saddle $(0,0)$.
Surrounding $(1,0)$ there are a continuous family of  periodic orbits $\Gamma_{h},\,h\in(-\frac{1}{20},0)$, 
which are bounded by the homoclinic loop $\Gamma_{0}$ passing through the origin.

For smooth polynomial perturbations of system \eqref{hsys}, we consider the following near-Hamiltonian system
\begin{equation}\label{sys1}
\dot{x} = y + \sum_{k=1}^{+\infty} \varepsilon^{k} P_k(x,y),\quad \dot{y} = x^3(1-x) +  \sum_{k=1}^{+\infty} \varepsilon^{k} Q_k(x,y),
\end{equation}
where $|\varepsilon| \ll 1$,  $P_k(x,y)$ and $Q_k(x,y)$ are quartic polynomials given by
$$P_k(x,y)=\sum_{i+j=0}^{4} a_{ijk} x^i y^j,\quad Q_k(x,y)=\sum_{i+j=0}^{4} b_{ijk} x^i y^j$$
with the coefficients $a_{ijk}$ and $b_{ijk}$ as free parameters.

Then a bifurcation function $d(h,\varepsilon)$ of system \eqref{sys1} has the following expansion
\begin{equation}\label{g1}
d(h,\varepsilon)=\varepsilon M_1(h)+\varepsilon^2 M_2(h)+\cdots +\varepsilon^k M_k(h)+\cdots,\quad |\varepsilon| \ll 1,
\end{equation}
 where
\begin{equation*}
M_1(h)=\oint_{\Gamma_h}Q_1(x,y)\mathrm{d}x-P_1(x,y)\mathrm{d}y,\quad h\in\Big(-\frac{1}{20},0\Big).
\end{equation*}
It is well known that one simple zero of the first nonvanishing Melnikov function $M_k(h)$ in \eqref{g1} for $0<-h\ll 1$ corresponds to a limit cycle
of system \eqref{sys1} produced near $\Gamma_0$ for $|\varepsilon|$ sufficiently small.

By studying $M_1(h)$, we get the following theorem.

\begin{thm}\label{thm1}
Let \eqref{g1} hold.  If $M_1(h)\not\equiv0$, there exists $0 < \varepsilon_0\ll 1$ such
that  $M_1(h)$ has at most $5$ zeros (counting multiplicity) on the interval $h\in(-\varepsilon_0, 0)$.
This upper bound can be reached for simplie zeros with proper values of parameters.
Furthermore, $M_1(h)\equiv0$ if and only if
\begin{equation}\label{g2}
\begin{split}
 a_{121 }&= -3 b_{031},\quad  b_{111 }= -2 a_{201},\quad a_{101}= -b_{011},\\
 3b_{131}&= -2 a_{221},\quad  b_{211 }= -3 a_{301},\quad  b_{311 }= -4 a_{401}.
\end{split}
\end{equation}
\end{thm}

For $M_1(h)\equiv0$, by $M_2(h)$ we have the following theorem.

\begin{thm}\label{thm2}
Let \eqref{g1} hold. If $M_1(h)\equiv0$ and $M_2(h)\not\equiv0$, then 
$M_2(h)$ has at most $11$ zeros (counting multiplicity) for $0<-h\ll 1$.
This upper bound can be reached by proper perturbations in system \eqref{sys1}.
\end{thm}

It is difficult to study the asymptotic expansion of $M_3(h)$ in \eqref{g1} for  $0<-h\ll 1$,
because the computation of its coefficients involves too many parameters 
$a_{ijk}$ and $b_{ijk}$, $k=1,2,3$.


For piecewise smooth perturbations of system \eqref{hsys}, we study perturbations in the form of system \eqref{e1}.
To be precise, we consider the following system  
\begin{equation}\label{sys}
\begin{split}
(\dot{x},\dot{y})=\left\{
\begin{aligned}
(y+\varepsilon \widetilde P^+(x,y),\,\,&x^3(1-x)+\varepsilon \widetilde Q^+(x,y)) ,\quad (x,y)\in\Sigma^+,\\[1ex]
(y+\varepsilon \widetilde P^-(x,y),\,\,&x^3(1-x)+\varepsilon \widetilde Q^-(x,y)) ,\quad (x,y)\in\Sigma^-,
\end{aligned}
\right.
\end{split}
\end{equation}
where $|\varepsilon|\ll1$,
\begin{equation*}\label{sp1}
\Sigma^\pm=\{(x,y)\in \mathbb{R}^2\,|\,\pm( \sin(\theta)(x-1)-\cos(\theta) y)>0\},\,0<\theta\leq\pi,
\end{equation*}
and $\widetilde P^\pm(x,y)$ and $\widetilde Q^\pm(x,y)$ are quartic polynomials in $(x,y)$.
In system \eqref{sys} the separation line passes through the center $(1,0)$,
and has two intersection points with $\Gamma_h$ for each $h\in(\frac{1}{20},0]$. 

In order to find limit cycles bifurcating near $\Gamma_0$ in system \eqref{sys}, we shall investigate 
the number of isolated zeros of Melnikov function
\begin{equation}\label{g3}
\widetilde{M}(h)=\,\int_{\Gamma_h^+}\widetilde Q^+(x,y)\mathrm{d}x-\widetilde P^+(x,y)\mathrm{d}y 
+\int_{\Gamma_h^-}\widetilde Q^-(x,y)\mathrm{d}x-\widetilde P^-(x,y)\mathrm{d}y
\end{equation}
for $0<-h\ll 1$, where $\Gamma^\pm_{h}=\Gamma_h\cap\Sigma^\pm$.
By studying the coefficients of the corresponding asymptotic expansion of $\widetilde M(h)$, 
we have the following theorem.

\begin{thm}\label{th3}
Let \eqref{g3} hold for system \eqref{sys}. If $\widetilde{M}(h)\not\equiv0$, there exists $0<\varepsilon_0\ll 1$
such that on the interval $h\in(-\varepsilon_0,0)$ $\widetilde{M}(h)$ has at most $19$ zeros (counting multiplicity) for $\theta=\arctan\frac{121\sqrt{6}}{100}$,
or at most $14$ zeros (counting multiplicity) for $\theta=\frac{\pi}{2}$, or at most $12$ zeros (counting multiplicity) for $\theta=\pi$.
This upper bound is sharp.
\end{thm}



The paper is organized as follows:
In Section \ref{sec2}, we obtain the explicit Picard-Fuchs equations for Abelian integrals $I_{i,j}(h)$ and
present the proof for Theorems \ref{mth1} and \ref{st2} to extend the method of Picard-Fuchs equations
to piecewise smooth polynomial perturbations.
In Section \ref{sec3}, we discuss the method of using Picard-Fuchs equations to compute the asymptotic expansion of generating functions near a homoclinic loop. 
In Sections \ref{sec4} and \ref{sec5}, we shall prove Theorems \ref{thm1}, \ref{thm2} and \ref{th3}, respectively.


\section{Picard-Fuchs equations of integrals $I_{i,j}(h)$ and $I_{i,j}^+(h)$}\label{sec2}
In this section, we assume that the Hamiltonian $H(x,y)$ has the following form
\begin{equation}\label{b1a}
H(x,y)=\frac{1}{2}y^2+b_1x + b_2x^2 + \cdots + b_n x^n,\,\,b_n\neq0,\,\,n\geq2.
\end{equation}
For Abelian integrals $I_{i,j}(h)$ and $I_{i,j}^+(h)$ in \eqref{h4a}, we shall present algebraic structures
and  the explicit expression of  the related Picard-Fuchs equations in the following two subsections, respectively.

\subsection{Homogeneous Picard-Fuchs equations of integrals $I_{i,j}(h)$}
In 1984, Petrov \cite{Petrov1984} proved that for the Hamiltonian \eqref{b1a}
any Abelian integral $I(h)$ along periodic orbits $\Gamma_h$ can be expressed as
\begin{equation}\label{g4}
I(h)=\sum_{k=0}^{n-2}p_k(h)I_{k,1}(h),
\end{equation}
where $p_k(h)$, $k=0,1,\ldots,n-2$ are polynomials in $h$. 
A procedure is shown in \cite{NY2001} for computing  \eqref{g4} and the related Picard-Fuchs equations
by using Gelfand-Leray residue.
For conveniece of the reader, we give the following lemma for \eqref{g4}
and the explicit expresion of Picard-Fuchs equations in Theorem \ref{lmm2}.

\begin{lm}\label{lmm1}
Let \eqref{b1a} hold.
For Abelian integrals $I_{i,j}(h)$ in \eqref{h4a} we have:\\[1ex]
{\rm(i)} $I_{i,j}(h)\equiv0$ for $j$ even;\\[1ex]
{\rm(ii)} for $i\ge-1$ and $j\ge -1$,
\begin{center}
$(i+1)I_{i,j+2}(h)=(j+2)\sum\limits_{k=1}^{n}kb_kI_{i+k,j}(h);$
\end{center}
{\rm(iii)}  for $i\ge -1$,
\begin{center}
$(2i+3n+2)b_nI_{i+n,1}(h)=(2i+2)hI_{i,1}(h)- \sum\limits_{k=1}^{n-1}(2i+3k+2)b_kI_{i+k,1}(h).$
\end{center}
\end{lm}
\begin{proof}
Because every periodic orbit $\Gamma_h$ of system $\mathrm{d}H=0$  
is symmetric with respect to the $x$-axis, we can easily get the statement (\textrm{i}).

(ii) Because $y\mathrm{d}y + (b_1+2b_2x +\cdots + n b_n x^{n-1})\mathrm{d}x=0$ along $\Gamma_h$, for $i\ge -1$ and $j\ge -1$ we obtain
\begin{equation}\label{g5}
\oint_{\Gamma_h}x^{i+1}y^{j+1}\mathrm{d}y +\oint_{\Gamma_h} (b_1+2b_2x +\cdots + n b_n x^{n-1})x^{i+1}y^j\mathrm{d}x=0.
\end{equation}
On the other hand, by Green's Theorem we have
\begin{equation}\label{g6}
\oint_{\Gamma_h}x^{i+1}y^{j+1}\mathrm{d}y=-\frac{i+1}{j+2}I_{i,j+2}(h).
\end{equation}
Then substituting \eqref{g6} into \eqref{g5} yields the statement (ii).

(iii) Because $y^2=2h-2b_1x - 2b_2x^2 - \cdots - 2b_nx^n$ along $\Gamma_h$,
for $i\ge -1$ and $j\ge -1$ we have
\begin{equation}\label{g7}
I_{i,j+2}(h)=2hI_{i,j}(h)-2b_1I_{i+1,j}(h)-2b_2I_{i+2,j}(h)-\cdots-2b_nI_{i+n,j}(h).
\end{equation}
Then eliminating $I_{i,j+2}(h)$ from the statement (ii) and \eqref{g7} we get
\begin{equation}\label{g8}
\begin{split}
-2(i+1)hI_{i,j}+(2i+j+4)b_1I_{i+1,j}
&+(2i+2j+6)b_2I_{i+2,j}\\
+\cdots&+(2i+nj+2n+2)b_nI_{i+n,j}=0.
\end{split}
\end{equation}
Taking $j=1$ for \eqref{g8} we can get the statement (iii).
\end{proof}

By Lemma \ref{lmm1}, we can get the explicit expression of a linear combination of 
$I_{0,1}(h)$, $I_{1,1}(h)$, $\ldots$, $I_{n-2,1}(h)$ for any Abelian integral $I_{i,j}(h)$, $i+j\ge 1$.
The related Picard-Fuchs equations are given in the next theorem.

\begin{thm}\label{lmm2}
Let \eqref{b1a} hold and $\mbox{\boldmath $X$}(h)=\mbox{\rm col}(I_{0,1}(h),\,I_{1,1}(h),\,\ldots,\,I_{n-2,1}(h))$. Then
the column vector function $\mbox{\boldmath $X$}(h)$ satisfies the Picard-Fuchs equations
\begin{equation}\label{g9}
2(\mbox{\boldmath $T$}_1+\mbox{\boldmath $T$}_2\mbox{\boldmath $T$}_4^{-1}\mbox{\boldmath $T$}_3)\mbox{\boldmath $X$}'(h)=\mbox{\boldmath $X$}(h),
\end{equation}
where  $'$ represents the derivative with respect to $h$, and
\begin{equation}\label{st3}
\begin{split}
\mbox{\boldmath $T$}_1 =&
\left( \begin{array}{cccc}
h & -b_1 & \cdots & -b_{n-2} \\
0 & h & \cdots & -b_{n-3} \\
\vdots & \vdots & \ddots & \vdots \\
0 & 0 & \cdots & h
\end{array} \right),\qquad\quad
\mbox{\boldmath $T$}_2 =
\left( \begin{array}{ccccc}
b_{n-1} & b_n & \cdots &0& 0 \\
b_{n-2} & b_{n-1} & \cdots & 0& 0 \\
\vdots & \vdots & \ddots&\vdots & \vdots \\
b_1 & b_2 & \cdots & b_{n-1} &b_n
\end{array} \right),\\
\mbox{\boldmath $T$}_3 =&
\left( \begin{array}{cccc}
b_1 & 2b_2 & \cdots & (n-1)b_{n-1} \\
-2h & 3b_1 & \cdots & nb_{n-2} \\
0   & -4h & \cdots  & (n+1)b_{n-3} \\
\vdots & \vdots & \ddots & \vdots \\
0 & 0 & \cdots & -2(n-1)h
\end{array} \right),\\
\mbox{\boldmath $T$}_4 = &
\left( \begin{array}{cccc}
nb_{n} & 0 & \cdots & 0 \\
(n+1)b_{n-1} & (n+2)b_{n} & \cdots & 0 \\
\vdots & \vdots & \ddots&\vdots  \\
(2n-1)b_1 & 2nb_2 & \cdots & (3n-2)b_n
\end{array} \right).
\end{split}
\end{equation}
\end{thm}

\begin{proof}
Note that $I_{i,-1}(h)=I'_{i,1}(h)$ for $i\geq0$.
Taking $j=-1$ and $i=0,1,\ldots,n-2$ for \eqref{g7} we get $n-1$ equations, which can be written as
\begin{equation}\label{g10}
2(\mbox{\boldmath $T$}_1\mbox{\boldmath $X$}'(h)-\mbox{\boldmath $T$}_2\mbox{\boldmath $Y$}'(h))=\mbox{\boldmath $X$}(h),
\end{equation}
where $\mbox{\boldmath $Y$}(h)=\mbox{\rm col}(I_{n-1,1}(h),\,I_{n,1}(h),\,\ldots,\,I_{2n-2,1}(h))$.
Similarly, taking $i=-1,\ldots,n-2$ and $j=-1$ for \eqref{g8} we get
\begin{equation}\label{g11}
\mbox{\boldmath $T$}_3\mbox{\boldmath $X$}'(h)+\mbox{\boldmath $T$}_4\mbox{\boldmath $Y$}'(h)=0.
\end{equation}
Then eliminating $\mbox{\boldmath $Y$}'(h)$ from \eqref{g10} and \eqref{g11} we get \eqref{g9}.
\end{proof}


\subsection{Proof of Theorems \ref{mth1} and \ref{st2}}
In this subsection, to prove Theorems \ref{mth1} and \ref{st2} for piecewise smooth system \eqref{e1} 
we extend the method of computing the algebraic structure of Abelian integrals
and Picard-Fuchs equations for smooth perturbations in the previous subsection.

Note that $\widetilde{M}(h)$ in \eqref{e2} can be rewritten as
\begin{equation}\label{xx1}
\widetilde{M}(h)= \oint_{\Gamma_h}Q^-(x,y)\mathrm{d}x - P^-(x,y)\mathrm{d}y + \widetilde M^+(h),
\end{equation}
where
\begin{equation*}
\widetilde M^+(h)=\int_{\Gamma_h^+} \big(Q^+(x,y)-Q^-(x,y)\big) \mathrm{d}x - \big(P^+(x,y)-P^-(x,y)\big)\mathrm{d}y.
\end{equation*}
By Lemma \ref{lmm1}, the first integral in \eqref{xx1} can be expressed as a linear combination of Abelian integrals $I_{j,1}(h)$, $j=0,1,\ldots,n-2$.
For the second integral $\widetilde M^+(h)$, we have
\begin{equation}\label{h4}
\begin{split}
\int_{\Gamma^+_h} x^iy^j\mathrm{d}y =& \int_{\Gamma^+_h} x^i\mathrm{d}\Big(\frac{1}{j+1}y^{j+1}\Big)\\
=& - \frac{i}{j+1} I^+_{i-1,j+1}(h) + \frac{1}{j+1}K_{i,j+1}(h),
\end{split}
\end{equation}
where $I_{i-1,j+1}^+(h)$ and $K_{i,j+1}(h)$ are given by \eqref{h4a}.
Then $\widetilde M^+(h)$ can be expressed as a linear combination of integrals $I^+_{i,j}(h)$ and $K_{i,j+1}(h)$
for $1\le i+j\le n$.


To prove Theorem \ref{mth1}, we need to study the algebraic structure for integrals $I^+_{i,j}(h)$ and $K_{i,j}(h)$, respectively.
For Abelian integrals $I^+_{i,j}(h)$ we have the following lemma.

\begin{lm} \label{st1}
Let \eqref{h4a} and \eqref{b1a} hold.
Then for Abelian integrals $I^+_{i,j}(h)$ the following identities hold:
\begin{itemize}
\item[{\rm(i)}] for $i\ge-1$ and $j\ge -1$,
\begin{center}
$(i+1)I^+_{i,j+2}(h)=(j+2)\sum\limits_{k=1}^{n}kb_kI_{i+k,j}^+(h)+K_{i+1,j+2}(h)$;
\end{center}
\item[{\rm(ii)}]  for $i\ge -1$ and $j=0,1$,
\begin{equation*}
\begin{split}
&\,(2i +nj + 2n+2)b_nI^+_{i+n,j}(h)\\
=&\, (2i+2)hI^+_{i,j}(h)- \sum\limits_{k=1}^{n-1}(2i+kj+2k+2)b_kI_{i+k,j}^+(h)- K_{i+1,j+2}(h).
\end{split}
\end{equation*}
\end{itemize}
\end{lm}

\begin{proof}
(i) Because $y\mathrm{d}y + (b_1+2b_2x +\cdots + n b_n x^{n-1})\mathrm{d}x=0$ along $\Gamma_h^+$, for $i\ge -1$ and $j\ge -1$ we obtain
\begin{equation}\label{t1}
\int_{\Gamma^+_h}x^{i+1}y^{j+1}\mathrm{d}y +\int_{\Gamma^+_h} (b_1+2b_2x +\cdots + n b_n x^{n-1})x^{i+1}y^j\mathrm{d}x=0.
\end{equation}
On the other hand, by \eqref{h4} we directly have
\begin{equation}\label{t2}
\int_{\Gamma^+_h}x^{i+1}y^{j+1}\mathrm{d}y=-\frac{i+1}{j+2}I^+_{i,j+2}(h) + \frac{1}{j+2}K_{i+1,j+2}(h).
\end{equation}
Then substituting \eqref{t2} into \eqref{t1} yields the statement (i).

(ii) Because $y^2=2h-2b_1x - 2b_2x^2 - \cdots - 2b_nx^n$ along $\Gamma^+_h$,
for $i\ge -1$ and $j\ge -1$ we have
\begin{equation}\label{t3}
I^+_{i,j+2}(h)=2hI^+_{i,j}(h)-2b_1I^+_{i+1,j}(h)-2b_2I^+_{i+2,j}(h)
-\cdots-2b_nI^+_{i+n,j}(h).
\end{equation}
Then eliminating $I^+_{i,j+2}(h)$ by the statement (i) and \eqref{t3} we get
\begin{equation}\label{t4}
\begin{split}
-2(i+1)hI^+_{i,j}+(2i&+j+4)b_1I^+_{i+1,j}
+(2i+2j+6)b_2I^+_{i+2,j}\\
+\cdots&+(2i+nj+2n+2)b_nI^+_{i+n,j} + K_{i+1,j+2}=0,
\end{split}
\end{equation}
which yields the statement (ii) for $j=0,1$.
\end{proof}

For the algebraic structure of functions $K_{i,j}(h)$, we have the following lemma.

\begin{lm}\label{lm3}
Let \eqref{h4a} and \eqref{b1a} hold. Suppose that the separation line of system \eqref{e1} is given by $y=k(x-c)$, where $k\neq0$. Then the following identities hold:

\rm(\romannumeral1) $k^iK_{i,j}(h)=\sum\limits_{s=0}^i
\left(\!\!\begin{array}{c}i\\ s \end{array}\!\!\right)
(kc)^s K_{0,i+j-s}(h)$ for $i\geq1$;

\rm(\romannumeral2) $b_nK_{0,j+n}(h)= k^nhK_{0,j}(h) - \frac{1}{2}k^nK_{0,j+2}(h)-\sum\limits_{i=0}^{n-1}f_ik^{n-i}K_{0,j+i}(h)$ for $j\geq0$,
where $f_i=\frac{1}{i!}\frac{\partial^i}{\partial x^i}H(c,0)$, $1\le i\le n-1$, and $f_0=H(c,0)$.
\end{lm}

\begin{proof}
Let $(x_1(h),y_1(h))$ to $(x_2(h),y_2(h))$ be the starting point and the ending point of the orbit arc $\Gamma^+_h$, respectively.
Then $K_{i,j}(h)= x_2^i(h)y_2^j(h)-x_1^i(h)y_1^j(h)$ by \eqref{h4a}.
Since $y_2(h)=k(x_2(h)-c)$ and $y_1(h)=k(x_1(h)-c)$, for $i\geq1$ we have
\begin{equation*}
\begin{split}
k^iK_{i,j}(h)
&=\, k^i(x_2^i(h)y_2^j(h)-x_1^i(h)y_1^j(h))\\
&=\, (y_2(h)+kc)^iy_2^j(h)-(y_1(h)+kc)^iy_1^j(h).\\
\end{split}
\end{equation*}
Then expanding the equation above yields the statement (i).

By $y_2(h)=k(x_2(h)-c)$, it is easy to get that
\begin{equation}\label{h8a}
\begin{split}
h &=\, H(x_2(h),y_2(h))=\, \frac{1}{2}y_2^2(h) + \sum\limits_{i=0}^n f_i(x_2(h)-c)^i\\
&=\, \frac{1}{2}y_2^2(h) + \sum\limits_{i=0}^n f_ik^{-i}y_2^i(h)\\
&=\, \frac{1}{k^n}\Big(\frac{1}{2}k^ny_2^2(h) + \sum\limits_{i=0}^n f_ik^{n-i}y_2^i(h)\Big).
\end{split}
\end{equation}
Then multiplying $y_2^j(h)$ on both sides of \eqref{h8a} yields
\begin{equation}\label{h8}
\begin{split}
\frac{1}{2}k^ny_2^{j+2}(h) + \sum\limits_{i=0}^n f_ik^{n-i}y_2^{j+i}(h) = k^nhy_2^{j}(h),\quad j\geq0.
\end{split}
\end{equation}
Similarly, from $h=H(x_1(h),y_1(h))$  and $y_1(h)=k(x_1(h)-c)$, we have
\begin{equation}\label{h9}
\begin{split}
\frac{1}{2}k^ny_1^{j+2}(h) + \sum\limits_{i=0}^n f_ik^{n-i}y_1^{j+i}(h) = k^nhy_1^{j}(h)\quad j\geq0.
\end{split}
\end{equation}
Note that $f_n=b_n$. Subtracting \eqref{h9} from \eqref{h8} yields the statement \rm(\romannumeral2).
\end{proof}

Note that the separation line in system \eqref{e1} is given by $L(x,y,\theta)=0$, where $L(x,y,\theta)=\sin(\theta)(x-c)-\cos(\theta) y$, $\theta\in(0,\pi]$.
By Lemma \ref{lm3}, any $K_{i,j}(h)$ can be expressed as
\begin{equation*}
K_{i,j}(h)=\sum_{s=1}^{n-1}r_{i,j,s}(h)K_{0,s}(h)\,\, \mbox{for}\,\,\theta\in\Big(0,\frac{\pi}{2}\Big)\cup\Big(\frac{\pi}{2},\pi\Big),
\end{equation*}
where $r_{i,j,s}(h)$, $1\le s\le n-1$, are polynomials in $h$ with $\mbox{deg}(r_s)\le (i+j-s)/n$.
For the case of $\theta = \frac{\pi}{2}$ (i.e. the separation line is $x=c$), we have
\begin{equation}\label{q4}
\begin{split}
&{\rm(\romannumeral1)}\,K_{i,j}(h)\equiv0\,\,{\rm for} \,\, j\,\, {\rm even};\\
&{\rm(\romannumeral2)}\,  K_{i,j}(h)=c^iK_{0,j}(h) \,\,{\rm for}\,\,i\geq1;\\
&{\rm(\romannumeral3)}\, K_{0,j+2}(h)=2\Big(h-\sum\limits_{i=1}^{n}b_ic^i\Big)K_{0,j}(h) \,\,{\rm for}\,\,j\geq0.
\end{split}
\end{equation}
For the case of $\theta = \pi$ (i.e. the separation line is $y=0$), we have
\begin{equation}\label{q3}
\begin{split}
&{\rm(\romannumeral1)}\,K_{i,j}(h)\equiv0 \,\,{\rm for}\,\,j\geq1;\\
&{\rm(\romannumeral2)}\,b_nK_{n+j,0}(h)= hK_{j,0}(h) - \sum\limits_{i=1}^{n-1}b_iK_{i+j,0}(h) \,\,{\rm for}\,\,j\geq0.
\end{split}
\end{equation}
Identities in \eqref{q4} and \eqref{q3} can be similarly proved as in the proof of Lemma \ref{lm3}.

\begin{proof}[Proof of Theorem \ref{mth1}]
By Lemma \ref{st1}, any Abelian integral $I^+_{i,j}(h)$ can be expressed as a linear combination
 of $I^+_{k,1}(h)$, $0\le k\le n-2$ and $K_{r,s}(h)$ with $2r+ns \le 2i+nj+2$, where by \eqref{h4a} we have taken into account
\begin{equation}\label{q6}
I^+_{k,0}(h) = \frac{1}{k+1}K_{k+1,0}(h).
\end{equation}
By Lemma \ref{lm3}, \eqref{q4} and \eqref{q3}, for any $\theta\in(0,\pi]$ any $K_{i,j}(h)$ can be expressed as
\begin{equation*}
K_{i,j}(h)=\sum_{s=1}^{n-1}\tilde r_{i,j,s}(h)\widetilde K_{s}(h), 
\end{equation*}
where coefficients $\tilde r_{i,j,s}(h)$ are polynomials in $h$.
Then further using Lemma \ref{lmm1} and \eqref{h4}, $\widetilde{M}(h)$ in \eqref{xx1} can be written
into the form \eqref{h4b}.
\end{proof}

Next, we give the proof of Theorem \ref{st2}.

\begin{proof}[Proof of Theorem \ref{st2}]
Note that $I^+_{i,-1}(h)=\frac{{\mathrm d}}{{\mathrm d}h}I^+_{i,1}(h) - J_i(h)$ for $i\geq0$.
Taking $j=-1$ and $i=0,1,\ldots,n-2$ for \eqref{t3} we get $n-1$ equations which can be written as
\begin{equation}\label{t6}
2\mbox{\boldmath $T$}_1(\widetilde{\mbox{\boldmath $X$}}'(h)-\mbox{\boldmath $J$}(h)) - 2\mbox{\boldmath $T$}_2\widetilde{\mbox{\boldmath $Y$}}(h) = \widetilde{\mbox{\boldmath $X$}}(h),
\end{equation}
where
\begin{equation*}
\begin{split}
\widetilde{\mbox{\boldmath $Y$}}(h)=
\, &  \mbox{\rm col}\Big(\frac{{\mathrm d}}{{\mathrm d}h}I^+_{n-1,1}(h)-J_{n-1}(h),
\ldots,\frac{{\mathrm d}}{{\mathrm d}h}I^+_{2n-2,1}(h)-J_{2n-2}(h)\Big).
\end{split}
\end{equation*}
Similarly, taking $j=-1$ and $i=-1,0,\ldots,n-2$ for \eqref{t4} we get
\begin{equation}\label{t7}
\mbox{\boldmath $T$}_3(\widetilde{\mbox{\boldmath $X$}}'(h) - \mbox{\boldmath $J$}(h)) + \mbox{\boldmath $T$}_4\widetilde{\mbox{\boldmath $Y$}}(h) + \mbox{\boldmath $K$}(h)=0.
\end{equation}
Then eliminating $\widetilde{\mbox{\boldmath $Y$}}(h)$ from \eqref{t6} and \eqref{t7} we get \eqref{t5}.
\end{proof}

\section{Asymptotic expansion of Abelian integrals}\label{sec3}
In this section, we suppose that $\Gamma_{\beta}$ is a homoclinic loop of the unperturbed system \eqref{e1}$|_{\varepsilon=0}$.
We shall discuss the method of applying Picard-Fuchs equations 
to compute the asymptotic expansion of Melnikov function 
$\widetilde{M}(h)$ near $\Gamma_{\beta}$ for the piecewise smooth system \eqref{e1} 
and how to determine the maximum number of isolated zeros of $\widetilde M(h)$ for $0<\beta-h\ll1$. 
As an illustration, we study the asymptotic expansion of Melnikov function near a homoclinic hoop 
for system \eqref{sys}.

\subsection{Methodology}
In order to compute the asymptotic expansion of $\widetilde M(h)$ 
near $\Gamma_{\beta}$ for system \eqref{e1},
by \eqref{h4b} we only need to study the asymptotic expansions of integrals $I_{i,1}(h)$, $I_{i,1}^{+}(h)$
and $\widetilde{K}_{i}(h)$ for  $0\le\beta-h\ll1$, $i=0,1,\ldots$, $n-2$.

By \eqref{h4a} we can see that the asymptotic expansion of $\widetilde K_i(h)$ 
requires the corresponding asymptotic expansions of $x_j(h)$ and $y_j(h)$ for  $0\le\beta-h\ll1$, 
where points $(x_j(h),y_j(h))$, $j=1,2$, satisfy $H(x_j(h),y_j(h))=h$
and $L(x_j(h),y_j(h),\theta)=0$.
Using the implicit function theorem,  we can see that $x_j(h)$ and $y_j(h)$
are analytic functions in $|h-\beta|^{\frac{1}{p}}$,
where $p$ is a positive integer depending on $H$ and the value of $\theta$.
Therefore, the variation of the value of $\theta$ for the separation line
may lead to changes in the form of the asymptotic expansion of $\widetilde M(h)$ near $\Gamma_\beta$
through integrals $\widetilde K_i(h)$.
See the next section for example.

For intergrals $I_{i,1}(h)$, we can use Picard-Fuchs equations \eqref{g9} to compute their asymptotic expansions near $h=\beta$. 
Note that \eqref{g9} can be rewritten as
\begin{equation}\label{g12}
P(h)\mbox{\boldmath $X$}'(h)=\mbox{\boldmath $T$}(h)\mbox{\boldmath $X$}(h),\quad h \in(\alpha, \beta),
\end{equation}
where $P(h)={\rm det}\,(\mbox{\boldmath $T$}_1+\mbox{\boldmath $T$}_2\mbox{\boldmath $T$}_4^{-1}\mbox{\boldmath $T$}_3)$ 
is a polynomial in $h$ of degree at most $n-1$
 satisfying $P(\beta)=0$, 
and $\mbox{\boldmath $T$}(h)$ is a matrix whose entries are polynomials in $h$ of degree at most $n-2$. 
Note system \eqref{g12} can be well-defined at $h=\beta$. 
Then we can use \eqref{g12} to compute the asymptotic expansion of $\mbox{\boldmath $X$}(h)$ near $h=\beta$ with 
undetermined coefficients. 
For the form of the asymptotic expansion of  $\mbox{\boldmath $X$}(h)$, see  \cite{H2013} 
for a homoclinic loop passing through a hyperbolic saddle, a cusp or a nilpotent saddle.

For the asymptotic expansions of integrals $I_{i,1}^{+}(h)$, 
we can rewrite \eqref{t5} into the following form
\begin{equation}\label{xx2}
P(h)\widetilde{\mbox{\boldmath $X$}}'(h)=\mbox{\boldmath $T$}(h)\widetilde{\mbox{\boldmath $X$}}(h) + P(h)\mbox{\boldmath $J$}(h) - \mbox{\boldmath $S$}(h)\mbox{\boldmath $K$}(h),
\end{equation}
where $\mbox{\boldmath $S$}(h)$ is also a matrix polynomial in $h$ of degree at most $n-2$.
The asymptotic expansions of {\boldmath $J$}$(h)$ and {\boldmath $K$}$(h)$ for $0<\beta-h\ll 1$ can be obtained by using the corresponding asymptotic expansions of $x_j(h)$ and
$y_j(h)$, $j=1,2$. Then we can use \eqref{xx2} to compute the asymptotic expansion of $\widetilde{\mbox{\boldmath $X$}}(h)$ for $0\le \beta-h\ll 1$.

It follows that $2\widetilde{\mbox{\boldmath $X$}}(h)={\mbox{\boldmath $X$}}(h)$ for $\theta=\pi$
because of the symmetry of $\Gamma_h$ with respect to the $x$-axis.
Apparently,  Picard-Fuchs equations \eqref{xx2} become into \eqref{g12} in this case.

Without loss of generality, for $0<\theta<\pi$ 
we assume that $\Gamma_{\beta}^+$ does not
contain a singular point of the unperturbed system of  \eqref{e1}.
Then $\widetilde{\mbox{\boldmath $X$}}(h)$ is analytic at $h=\beta$ for $0<\theta<\pi$. Then we can expand $\widetilde{\mbox{\boldmath $X$}}(h)$ into the form
\begin{equation}\label{xx3}
\widetilde{\mbox{\boldmath $X$}}(h)=\tilde{\mbox{\boldmath $c$}}_0 + \tilde{\mbox{\boldmath $c$}}_1(h-\beta)+\cdots+\tilde{\mbox{\boldmath $c$}}_j(h-\beta)^j+\cdots,\quad 0\le \beta-h\ll1,
\end{equation}
where $\tilde{\mbox{\boldmath $c$}}_0=\widetilde{\mbox{\boldmath $X$}}(\beta)$ and $\tilde{\mbox{\boldmath $c$}}_1=\frac{\mathrm{d}}{\mathrm{d}h}\widetilde{\mbox{\boldmath $X$}}(\beta)$.
The remaining coefficients can be obtained by substituting \eqref{xx3} into \eqref{xx2} and comparing the coefficients of the expansions of both sides.

\begin{Rm}
The method above is applicable to the case of $\Gamma_{\beta}$ as a heteroclinic loop.
In this case $\widetilde{\mbox{\boldmath $X$}}(h)$ may not be analytic near $h=\beta$ even for $0<\theta<\pi$, because $\Gamma_{\beta}^+$ could probably 
contain singular points.
 In order to compute the asymptotic expansion of $\widetilde{\mbox{\boldmath $X$}}(h)$ near $\Gamma_{\beta}^+$,
  we need to pay more attentions to the possible forms of the expansion and the formulas of its first several coefficients
when we vary the value of $\theta$,
which we shall investigate in our next paper.
\end{Rm}

At the end of this subsection, we shall present a result on determining the maximum number of isolated zeros of $\widetilde M(h)$
for $0<\beta-h\ll 1$ by using the coefficients of its corresponding asymptotic expansion. Because the asymptotic expansion of $\widetilde M(h)$
near the homoclinic loop $\Gamma_\beta$ for piecewise smooth perturbations can have different forms 
(for example see Subsection \ref{SubSec43}), 
it is necessary for us to consider the asymptotic expansion of Melnikov functions in a general form.

Suppose that Melnikov function $M(h, \mbox{\boldmath $\delta$})$ can be expanded into the following form
\begin{equation}\label{h5a}
M(h,\mbox{\boldmath $\delta$})=c_0(\mbox{\boldmath $\delta$})
+\sum_{j=0}^{+\infty}\sum_{l=1}^{p+1} c_{pj+j+l}(\mbox{\boldmath $\delta$})g_l(h)h^j
\end{equation}
for $0<-h\ll 1$, where {\boldmath $\delta$} $\in \mathbb{R}^m$ represents a parameter vector, 
and
\begin{equation}\label{h5b}
g_l(h)\in\Big\{|h|^{\frac{q}{p}},\,|h|^{\frac{q+1}{p}},\,\ldots,\, |h|^{\frac{q+p-1}{p}},\,h\ln|h|\Big\},\quad 1\le q\le p
\end{equation}
for some $p\in \mathbb{N}^+$ and $g_{l}(h)=o(g_{l-1}(h))$.
To determine the maximum number of isolated zeros of $M(h,\mbox{\boldmath $\delta$})$ for $0<-h\ll1$ we have

\begin{thm}\label{th4}
Let \eqref{h5a} and \eqref{h5b} hold.
Let each coefficient $c_j(\mbox{\boldmath $\delta$})$  in \eqref{h5a} be linear in {\boldmath $\delta$}.
Suppose that there exist $k$ nonnegative integers $0\le i_1<i_2<\cdots<i_k$, such that for any coefficient $c_j(\mbox{\boldmath $\delta$})$
if $j\not\in \{i_1,\,i_2,\,\ldots,\,i_k\}\triangleq S_k$, we have
\begin{equation}\label{h3b}
c_j=O(|c_{i_1},\,c_{i_2},\,\ldots,\,c_{i_j^{\ast}}|),
\end{equation}
where $i_j^{\ast}=\max\{i_s\in S_k|i_s<j\}$, and particularly if $1\le j<i_1$, then $c_j(\mbox{\boldmath $\delta$})\equiv0$ for all $\mbox{\boldmath $\delta$}$.
If further 
\begin{equation}\label{h3c}
\mathrm{rank}\,\frac{\partial(c_{i_1},\,c_{i_2},\,\ldots,\,c_{i_k})}
{\partial(\delta_1,\,\delta_2,\,\ldots,\,\delta_m)}=k,
\end{equation}
then for any compact set $D\subset \mathbb{R}^m$, when $M(h,\mbox{\boldmath $\delta$})\not\equiv0$,
there exists $\varepsilon_0>0$ such that $M(h,\mbox{\boldmath $\delta$})$ has at most $k-1$ isolated zeros (counting multiplicity)
in $0<-h<\varepsilon_0$ for {\boldmath $\delta$} $\in D$. Furthermore, $M(h,\mbox{\boldmath $\delta$})$ can
have $k-1$ isolated simple zeros in an arbitrary neighborhood of $h=0$ for $h<0$ and some {\boldmath $\delta$}.
\end{thm}

\begin{Rm}
 By \cite{H2013}, for analytic smooth near-Hamiltonian systems near a homoclinic or heteroclinic loop passing through hyperbolic saddles, cusps or nilpotent saddles,
Melnikov function can be expanded into the form in \eqref{h5a}. Then Theorem \ref{th4} can be also applied for the cases mentioned above.
\end{Rm}

Using the idea presented in the proof of Theorem 2.3.2 in \cite{H2013}, we can prove the above theorem.
For convenience of the reader, we present the proof below.

\begin{proof}
By \eqref{h3b} and \eqref{h3c} for any $j\not\in S_k$,  $c_j$ should be a linear combination of $c_{i_1}$,
$c_{i_2}$, $\ldots$, $c_{i^\ast_j}$, and we can take {\boldmath $\Lambda$}$=(c_{i_1},\,c_{i_2},\,\ldots,\,c_{i_k})$
as a free parameter vector. Then $M(h,\mbox{\boldmath $\delta$})$ can be written as
$M(h,\mbox{\boldmath $\delta$})\triangleq\widehat M(h,\mbox{\boldmath $\Lambda$})$, which is linear in {\boldmath $\Lambda$}.
Then $\widehat M(h,\mbox{\boldmath $\Lambda$})$ can be expanded as
\begin{equation}\label{h3d}
\widehat M(h,\mbox{\boldmath $\Lambda$})=c_{i_1}f_1(h) + c_{i_2}f_2(h) + \cdots + c_{i_k} f_k(h),
\end{equation}
where $f_s(h)\in C^{\infty}$, $1\le s\le k$. Comparing \eqref{h3d} with \eqref{h5a} we have $f_1(h)\sim 1$ for $i_1=0$ and
 $f_s(h)\sim h^{j_s}g_{l_s}(h)$ for $1\le l_s=i_s-Nj_s\le N$. Then there exists a small interval $(\beta,0)$ such that
 $f_s(h)\neq0$, $h\in(\beta,0)$ for any $1\le s\le k$.

We rewrite $\widehat M(h,\mbox{\boldmath $\Lambda$})$ as $\widehat M(h,\mbox{\boldmath $\Lambda$})=f_1(h)M_1(h)$, where
\begin{equation*} 
M_1(h)=c_{i_1} + c_{i_2}f_{2,1}(h)+\cdots+c_{i_k}f_{k,1}(h),
\end{equation*}
with $f_{s,1}(h)=f_s(h)/f_1(h)\in C^{\infty}$ satisfying $f_{s,1}(h)=o(f_{s-1,1}(h))$.
Then $\widehat M(h,\mbox{\boldmath $\Lambda$})$ has at most $k-1$ zeros (counting multiplicity) for $0<-h\ll 1$ if and only if
$M_1(h)$ has at most $k-1$ zeros (counting multiplicity) for $0<-h\ll 1$. It suffices to prove that $\frac{\mathrm{d}M_1(h)}{\mathrm{d}h}$
has at most $k-2$ zeros (counting multiplicity) for $0<-h\ll 1$.

Note that $\frac{\mathrm{d}M_1(h)}{\mathrm{d}h}$ can be written as
$\frac{\mathrm{d}M_1(h)}{\mathrm{d}h}=f'_{2,1}(h)M_2(h)$, where
\begin{equation*}
M_2(h)=c_{i_2}+ c_{i_3}f_{3,2}(h)+\cdots+c_{i_k}f_{k,2}(h),
\end{equation*}
where $f_{s,2}(h)=f'_{s,1}(h)/f'_{2,1}(h)\in C^\infty$ and $f_{s,2}(h)=o(f_{s-1,2}(h))$.
Because $f'_{2,1}(h)\neq0$ for $h$ negative and sufficiently close to $0$, $\frac{\mathrm{d}M_1(h)}{\mathrm{d}h}$
has at most $k-2$ zeros  (counting multiplicity)
in $0<-h\ll 1$ if and only if $M_2(h)$ has at most $k-2$ zeros (counting multiplicity) for $0<-h\ll1$.

Since the functions $M_1(h)$ and $M_2(h)$ have the same form,
using mathematical induction on $k$ we can show that $\widehat M(h,\mbox{\boldmath $\Lambda$})$ has most $k-1$ zeros (counting multiplicity)
in $0<-h\ll 1$ for {\boldmath $\delta$} $\in D$.
The independence of functions $f_1(h),\,f_2(h),\,\ldots, \,f_k(h)$ implies the existence of $k-1$ simple isolated zeros of $M(h,\mbox{\boldmath $\Lambda$})$
for $0<-h\ll 1$. The proof is completed.
\end{proof}


\subsection{Asymptotic expansion of $\widetilde M(h)$ for system \eqref{sys}}
In this subsection, we shall compute the asymptotic expansion of 
Melnikov function $\widetilde M(h)$ in \eqref{g3} near the homoclinic loop $\Gamma_0$ for  system \eqref{sys},
which will be needed for the proof of Theorem \ref{th3}.

For system \eqref{sys}, we assume 
$$P^\pm(x,y)=\sum_{i+j=0}^{4} a^\pm_{ij} x^i y^j,\quad Q^\pm(x,y)=\sum_{i+j=0}^{4} b^\pm_{ij} x^i y^j,$$
with the coefficients $a^\pm_{ij}$ and $b^\pm_{ij}$ as free parameters.
By \eqref{xx1}, \eqref{h4} and \eqref{q6}, $\widetilde M(h)$ in \eqref{g3} can be simplified as
\begin{equation}\label{m1a}
\widetilde M(h)=\sum_{i=0}^3\sum_{j=1}^{4-i}a_{i,j}I_{i,j}(h) + \sum_{i=0}^3\sum_{j=1}^{4-i}b_{i,j}I_{i,j}^+(h)
+\sum_{i+j=1}^5 e_{i,j}K_{i,j}(h),
\end{equation}
where
\begin{equation*}
\begin{split}
a_{i,j}=\frac{i+1}{j}a^-_{i+1,j-1} + b^-_{i,j},\quad b_{i,j}=&\,\frac{i+1}{j}(a^+_{i+1,j-1}-a^-_{i+1,j-1}) + b^+_{i,j}- b^-_{i,j},\\
e_{i,0}= \frac{1}{i}(b^+_{i-1,0}- b^-_{i-1,0})\,\,\mbox{and}&\,\, e_{i,j}=-\frac{1}{j}(a^+_{i,j}-a^-_{i,j})\,\, \mbox{for}\,\, j\ge 1.
\end{split}
\end{equation*}
Then we can take all the coefficients $a_{i,j}$, $b_{i,j}$ and $e_{i,j}$ in \eqref{m1a} as independent parameters.

We further simplify the expression of $\widetilde M(h)$  into the form \eqref{h4b} by the algebraic structures of integrals
$I_{i,j}(h)$, $I^+_{i,j}(h)$ and $K_{i,j}(h)$. 
By Lemmas \ref{lmm1}  and \ref{st1}, the algebraic structures of integrals $I_{i,j}(h)$ and $I^+_{i,j}(h)$
are given in the following two corollaries, respectively.

\begin{cor}\label{lmm3}
For Abelian integrals $I_{i,j}(h)$ of Hamiltonian $H(x,y)=\frac{1}{2}y^2-\frac{1}{4}x^4+\frac{1}{5}x^5$, the following identities hold:\\[1ex]
{\rm(i)} $I_{i,j}(h)\equiv0$ for $j$ even;\\[1ex]
{\rm(ii)} $(i+1)I_{i,j}(h)=j(I_{i+5,j-2}(h)-I_{i+4,j-2}(h))$ for $i\ge-1$ and $j\ge 1$;\\[1ex]
{\rm(iii)}  $(4i+14)I_{i,1}(h)=(5i+10)I_{i-1,1}(h)+20(i-4)h\,I_{i-5,1}(h)$ for $i\ge 5$,
and $I_{4,1}(h)=I_{3,1}(h)$.
\end{cor}

\begin{cor}\label{lm1}
For Abelian integrals $I^+_{i,j}(h)$ of system \eqref{sys}, the following identities hold:

{\rm(i)} $(i+1)I^+_{i,j}(h)=j(I^+_{i+5,j-2}(h)-I^+_{i+4,j-2}(h)) + K_{i+1,j}(h)$ for $i\ge-1$ and $j\ge 1$;\vspace{1ex}

{\rm(ii)}  $(4i+10j+4)I^+_{i,j}(h)=20(i-4)h\,I^+_{i-5,j}(h)+5(i+2j)I^+_{i-1,j}(h) - 10K_{i-4,j+2}(h)$ for $i\ge 5$ and $j=0,1$.
\end{cor}

It is straightforward to obtain the algebraic structure of $K_{i,j}(h)$ for system \eqref{sys}
by Lemma \ref{lm3} and identities in \eqref{q4} and \eqref{q3}.
Then by Theorem \ref{mth1} we can write $\widetilde M(h)$ as a linear combination of
$I_{i,1}(h)$, $I^+_{i,1}(h)$ and $\widetilde K_i(h)$, $0\le i\le 3$,
with explicit expressions of coefficients in parameters $a_{i,j}$, $b_{i,j}$ and $e_{i,j}$.

Next, for system \eqref{sys} we shall compute the asymptotic expansions of generating functions 
$I_{i,1}(h)$, $I^+_{i,1}(h)$ and $\widetilde K_i(h)$, $0\le i\le 3$.

By \cite{HYX2012},  the vector function $\mbox{\boldmath $X$}(h)=\mbox{\rm col}(I_{0,1}(h),\,I_{1,1}(h),\,I_{2,1}(h),\, I_{3,1}(h))$
 can be expanded in the following form
\begin{equation}\label{b9}
\mbox{\boldmath $X$}(h)= \sum_{j=0}^{+\infty}
\mbox{\boldmath $a$}_j h^j
+ \sum_{j=1}^{+\infty}\mbox{\boldmath $b$}_j \ln\!|h|\,h^j
+ \sum_{j=0}^{+\infty}(\mbox{\boldmath $c$}_j {\lvert h \rvert}^\frac{3}{4}
+\mbox{\boldmath $d$}_j {\lvert h \rvert}^\frac{5}{4}) h^j
\end{equation}
for $0 < - h \ll1$, where $\mbox{\boldmath $a$}_j$, $\mbox{\boldmath $b$}_j$, $\mbox{\boldmath $c$}_j$ and $\mbox{\boldmath $d$}_j$ are vector coefficients. 
Then by Theorem \ref{lmm2}, $\mbox{\boldmath $X$}(h)$ satisfies Picard-Fuchs equations
\begin{equation}\label{b7}
P(h)\mbox{\boldmath $X$}'(h)=(\mbox{\boldmath $A$}_1h+\mbox{\boldmath $A$}_0)\mbox{\boldmath $X$}(h),
\end{equation}
where $P(h)=20h(20h+1)$, 
\begin{equation}\label{b8}
\mbox{\boldmath $A$}_0 =
\left( \begin{array}{cccc}
15 & 2 & 3 & -26 \\
0 & 20 & 3 & -26 \\
0 & 0 & 25 & -26 \\
0 & 0 & 0 & 0
\end{array} \right), \quad
\mbox{\boldmath $A$}_1 =
\left( \begin{array}{cccc}
280 & 0 & 0 & 0 \\
-20 & 360 & 0 & 0 \\
-20 & -40 & 440 & 0 \\
-20 & -40 & -60 & 520
\end{array} \right).
\end{equation}
Then substituting \eqref{b9} into \eqref{b7} we get recursive formula for coefficients in \eqref{b9},
which are given in the next lemma.

\begin{lm}\label{lmm5}
Let $\mbox{\boldmath $A$}_0$ and $\mbox{\boldmath $A$}_1$ be matrices given in \eqref{b8}, 
and $\mbox{\boldmath $b$}_0=\mbox{\boldmath $0$}$.
For the  vector coefficients in \eqref{b9} we have
\begin{equation}\label{b9a}
\begin{split}
&\mbox{\boldmath $A$}_0\mbox{\boldmath $a$}_0=0,\,\qquad
\mbox{\boldmath $A$}_0 \mbox{\boldmath $b$}_1= 20\mbox{\boldmath $b$}_1,\,\qquad
\mbox{\boldmath $A$}_0 \mbox{\boldmath $c$}_0= 15\mbox{\boldmath $c$}_0,\,\qquad
\mbox{\boldmath $A$}_0 \mbox{\boldmath $d$}_0= 25\mbox{\boldmath $d$}_0, \\[1ex]
&\mbox{\boldmath $A$}_{j,1}\mbox{\boldmath $b$}_j =
(400(j-1)\,\mbox{\boldmath $E$}_4-\mbox{\boldmath $A$}_1)\,\mbox{\boldmath $b$}_{j-1},\,\,j\ge 2,\\[1ex]
&\mbox{\boldmath $A$}_{j,1}\mbox{\boldmath $a$}_j=
(400(j-1)\,\mbox{\boldmath $E$}_4-\mbox{\boldmath $A$}_1)\,\mbox{\boldmath $a$}_{j-1}
+400 \mbox{\boldmath $b$}_{j-1}+ 20 \mbox{\boldmath $b$}_{j},\,\,j\ge 1, \\[1ex]
&\mbox{\boldmath $A$}_{j,2}\mbox{\boldmath $c$}_j=
(100(4j-1)\,\mbox{\boldmath $E$}_4-\mbox{\boldmath $A$}_1)\,\mbox{\boldmath $c$}_{j-1},\,\,j\ge 1,\\[1ex]
&\mbox{\boldmath $A$}_{j,3}\mbox{\boldmath $d$}_j=
(100(4j+1)\,\mbox{\boldmath $E$}_4-\mbox{\boldmath $A$}_1)\,\mbox{\boldmath $d$}_{j-1},\,\,j\ge 1,
\end{split}
\end{equation}
where  $\mbox{\boldmath $A$}_{j,k}=\mbox{\boldmath $A$}_0-(20j+\lambda_k)\mbox{\boldmath $E$}_4$ with $\lambda_1=0$, $\lambda_2=15$,
$\lambda_3=25$, and {\boldmath $E$}$_4$ is the $4\times 4$ identity matrix.
\end{lm}

Note that det($\mbox{\boldmath $A$}_{1,1}$)=0 and all other matrices $\mbox{\boldmath $A$}_{j,k}$ in \eqref{b9a} are invertible.
Then if we have the values of $\mbox{\boldmath $a$}_0$, $\mbox{\boldmath $a$}_1$, $\mbox{\boldmath $b$}_1$, $\mbox{\boldmath $c$}_0$ and $\mbox{\boldmath $d$}_0$,
we can get all the remaining coefficients in \eqref{b9} by using \eqref{b9a}.

Suppose that $\mbox{\boldmath $a$}_0$, $\mbox{\boldmath $a$}_1$, $\mbox{\boldmath $b$}_1$, $\mbox{\boldmath $c$}_0$ and $\mbox{\boldmath $d$}_0$ are given by
\begin{equation*}
\begin{split}
&\mbox{\boldmath $a$}_k=(a_{ik})_{4\times1},\,\,\,
\mbox{\boldmath $b$}_1=(b_{i1})_{4\times1},\,\,\,
\mbox{\boldmath $c$}_0=(c_{i0})_{4\times1},\,\,\,
\mbox{\boldmath $d$}_0=(d_{i0})_{4\times1}.
\end{split}
\end{equation*}
Solving the first four equations of \eqref{b9a} yields
\begin{equation}\label{b13c}
\begin{split}
&a_{20}=\frac{5}{6}a_{10},\quad
a_{30}=\frac{25}{33}a_{10},\quad a_{40}=\frac{625}{858}a_{10},\\[0.5ex]
&b_{21}=\frac{5}{2}b_{11},\quad b_{31}=b_{41}=0,\\
&c_{20}=c_{30}=c_{40}=0,\\
&d_{20}=\frac{10}{7}d_{10},\quad d_{30}=\frac{50}{21}d_{10},\quad d_{40}=0.
\end{split}
\end{equation}
Using the formulas presented in \cite{HYX2012}, we obtain
\begin{equation}\label{b13d}
\begin{split}
a_{10}=&\oint_{\Gamma_0}y\,\mathrm{d}x=2\int^{\frac{5}{4}}_0x^2\sqrt{\frac{1}{2}
-\frac{2}{5}x}\,\mathrm{d}x=\, \frac{25}{84}\sqrt{2},\\
b_{11}=& -\frac{1}{8}\widetilde r_{10}=-\frac{\sqrt{2}}{5},\\
c_{10}=&\,\widetilde{r}_{00}\widetilde A_{0}=
2\sqrt{2}\times 4^{\frac{1}{4}}\widetilde A_{0}=4\widetilde A_0,\\
d_{10}=&\,\widetilde{r}_{20}\widetilde A_2
=\frac{84}{25}\widetilde A_2,
\end{split}
\end{equation}
where
\begin{equation*}\label{b13e}
\begin{split}
&\widetilde A_0= -\frac{2}{3}\int_0^1\frac{\mathrm{d}v}{\sqrt{1-v^4}}<0,\\
&\widetilde A_2= \frac{2}{5}\!\left[1-\int_0^1\frac{v^2\mathrm{d}v}{\sqrt{1-v^4}(1+\sqrt{1-v^4})}\right]\!>0.
\end{split}
\end{equation*}
With \eqref{b13c} and \eqref{b13d} holding, from
$\mbox{\boldmath $A$}_0 \mbox{\boldmath $a$}_1
+ \mbox{\boldmath $A$}_1 \mbox{\boldmath $a$}_{0}
= 20(\mbox{\boldmath $a$}_{1} + \mbox{\boldmath $b$}_{1})$
we get
\begin{equation}\label{b13f}
a_{21}=\frac{5}{2}a_{11}+3\sqrt{2},\quad
a_{31}=5\sqrt{2},\quad a_{41}=\frac{25}{6}\sqrt{2}.
\end{equation}
Then by \eqref{b9a}, \eqref{b13c}, \eqref{b13d}
and \eqref{b13f} we can derive the following lemma.

\begin{lm}\label{lmm6}
For system \eqref{sys1}, Abelian integrals $I_{0,1}(h)$, $I_{1,1}(h)$, $I_{2,1}(h)$ and
$I_{3,1}(h)$ have the following asymptotic expansions for $0<-h\ll 1$:
\begin{equation}
\begin{split}\label{b13a}
I_{0,1}(h)=\,& \frac{25}{84}\sqrt{2} + k_1 |h|^{\frac{3}{4}}-\frac{1}{5}\sqrt{2} h \ln\!|h| + k_2 h + k_3 |h|^{\frac{5}{4}}
- \frac{663}{1750} k_1 h |h| ^{\frac{3}{4}}\\
& + \frac{1386}{3125}\sqrt{2} h^2 \ln\!|h| +\cdots,\\
I_{1,1}(h)=\,& \frac{125}{504} \sqrt{2}-\frac{1}{2} \sqrt{2} h \ln\!|h| +(3 \sqrt{2}+\frac{5}{2} k_2) h
+\frac{10}{7} k_3 |h|^{\frac{5}{4}} - \frac{78}{175} k_1 h |h|^{\frac{3}{4}}\\
& + \frac{63}{125} \sqrt{2} h^2 \ln\!|h| +\cdots,\\
I_{2,1}(h)=\,& \frac{625}{2772} \sqrt{2}+5 \sqrt{2} h+\frac{50}{21} k_3 |h|^{\frac{5}{4}} - \frac{18}{35} k_1 h |h|^{\frac{3}{4}}
+\frac{14}{25} \sqrt{2} h^2 \ln\!|h| +\cdots,\\
I_{3,1}(h)=\,& \frac{15625}{72072} \sqrt{2}+\frac{25}{6} \sqrt{2} h - \frac{4}{7} k_1 h |h|^{\frac{3}{4}}
+ \frac{3}{5} \sqrt{2} h^2 \ln\!|h| +\cdots ,
\end{split}
\end{equation}
where $k_1$, $k_2$ and $k_3$ are constants with $k_1<0$ and $k_3>0$.
\end{lm}

For $\widetilde K_{i}(h)$, $i=0,1,2,3$,
we will use the implicit function theorem to compute the asymptotic expansions of $(x_j(h),y_j(h))$, $j=1,2$,
from
\begin{equation}\label{a5z}
\frac{1}{2}y^2-\frac{1}{4}x^4 + \frac{1}{5}x^5=h,\quad 
\sin(\theta)(x-1)-\cos(\theta) y=0.
\end{equation}

For $\theta=\pi$, we have $y_1(h)=y_2(h)\equiv0$. 
By the implicit function theorem, from $|h|^{\frac{1}{4}}=x(\frac{1}{4}-\frac{1}{5}x)^{\frac{1}{4}}$ it follows that $x_1(h)=\xi(|h|^{\frac{1}{4}})$ is analytic in $|h|^{\frac{1}{4}}$ with the following expansion
\begin{equation}\label{a6}
\begin{split}
x_1(h) =\,& \sqrt{2}|h|^{\frac{1}{4}} + \cdots + \frac{1067760993\sqrt{2}}{78125000}|h|^{\frac{11}{4}} + \cdots,\quad 0<-h\ll 1.
\end{split}
\end{equation}
Similarly, from $h=x^4(\frac{1}{4}-\frac{1}{5}x)$ it follows that $x_2(h)$ is analytic in $h$ with the following expansion
\begin{equation}\label{a7}
\begin{split}
x_2(h) =\,& \frac{5}{4} + \frac{256}{125}h + \cdots + \frac{6979321856}{48828125}h^{3} + \cdots, \quad 0<-h\ll 1.
\end{split}
\end{equation}

For $\theta=\frac{\pi}{2}$ from \eqref{a5z} we have
\begin{equation}\label{h12}
x_1(h)=x_2(h)=1,\quad y_1(h)=-y_2(h)=\sqrt{\frac{1+20h}{10}}.
\end{equation}

For $0<\theta\in(0,\frac{\pi}{2})\cup (\frac{\pi}{2},\pi)$, we assume that $\Gamma^+_0$ is from $(x_a,y_a)$ to $(x_b,y_b)$. 
Using $\tan(\theta)=\frac{y_a}{x_a-1}$ for \eqref{a5z}, 
it follows that $x_1(h)$ is analytic in $h$ with the following expansion

\begin{equation}\label{h2}
\begin{split}
x_1(h)=&\, x_a + \frac{10(x_a - 1)h}{x_a^{3}(6x_a^{2} - 15 x_a + 10)} - \frac{50(x_a - 1)N_1(x_a)h^{2}}{x_a^{7}(6x_a^{2} - 15x_a + 10)^3}\\
&+  \frac{500(x_a - 1)N_2(x_a)h^{3}}{x_a^{11}(6x_a^{2} - 15 x_a + 10)^{5}} + \cdots \triangleq \eta(h,x_a),\quad 0<-h \ll 1,
\end{split}
\end{equation}
where
\begin{equation*}
\begin{split}
N_1(x) &=\, 36x^{3} - 105x^{2} + 100 x - 30,\\
N_2(x) &=\, 1056x^6 - 6120x^{5} + 14645x^{4} - 18460x^{3} + 12880x^{2} - 4700 x + 700.
\end{split}
\end{equation*}
Similarly, we have
\begin{equation}\label{h3}
x_2(h)= \eta(h,x_b)\,\, \mbox{for}\,\, 0<-h \ll 1
\end{equation} by using $\tan(\theta)=\frac{y_b}{x_b-1}$ for \eqref{a5z}.

The separation line can be determined by $x_a$ with $0<x_a<\frac{5}{4}$ and $y_a>0$ for $k\triangleq\tan(\theta)\neq0$.
Here we only study the case of $1<x_a<\frac{5}{4}$,
because for $0<x_a<1$ (i.e.  $k<0$) we can get the previous case (i.e. $k>0$) by the transformation $(y,t)\rightarrow (-y,-t)$.
Note that $x_a$ and $x_b$ should satisfy
\begin{equation}\label{ha3}
(x_a-1)^2(5x_b^4-4x_b^5)-(x_b-1)^2(5x_a^4-4x_a^5)=0,
\end{equation}
which is obtained from
$$\frac{y_a}{x_a-1}=\frac{y_b}{x_b-1},\quad H(x_a,y_a)=H(x_b,y_b)=0.$$
Then for any $1<x_a<\frac{5}{4}$, the value of $x_b$ can determined for $0<x_b<1$ from \eqref{ha3},
even though its explicit expression is very complicated.

To get the asymptotic expansions of generating integrals $I^+_{0,1}(h)$, $I^+_{1,1}(h)$, $I^+_{2,1}(h)$, $I^+_{3,1}(h)$, we shall need the Picard-Fuchs equation,
which is given in the next lemma by Theorem \ref{st2}.

\begin{cor}\label{lm2}
Let $\widetilde{\mbox{\boldmath $X$}}(h)
=\mbox{\rm col}(I^+_{0,1}(h),\,I^+_{1,1}(h),\,I^+_{2,1}(h),\,I^+_{3,1}(h))$ for system \eqref{sys}. Then
$\widetilde{\mbox{\boldmath $X$}}(h)$ satisfies the Picard-Fuchs equation
\begin{equation}\label{h5}
P(h)\frac{\mathrm{d}}{\mathrm{d}h}\widetilde{\mbox{\boldmath $X$}}(h)=(\mbox{\boldmath $A$}_0 + \mbox{\boldmath $A$}_1h)\widetilde{\mbox{\boldmath $X$}}(h) + (\mbox{\boldmath $A$}_2 + \mbox{\boldmath $A$}_3h)\mbox{\boldmath $K$}(h) + P(h)\mbox{\boldmath $J$}(h),
\end{equation}
where $P(h)=20h(20h+1)$,
$$\mbox{\boldmath $K$}(h) = \mbox{\rm col}(K_{0,1}(h),\,\ldots,\,K_{4,1}(h)),\quad \mbox{\boldmath $J$}(h) = \mbox{\rm col}(J_0(h),\,\ldots,\,J_3(h)),$$
 $\mbox{\boldmath $A$}_0$ and $\mbox{\boldmath $A$}_1$ are given in \eqref{b8}, and
\begin{align*}
&\mbox{\boldmath $A$}_2 =
\left( \begin{array}{ccccc}
0 & -5 & -1 & -1 & 4 \\
0 &  0 & -5 & -1 & 4 \\
0 &  0 &  0 & -5 & 4 \\
0 &  0 &  0 &  0 & 0
\end{array} \right), \quad
\mbox{\boldmath $A$}_3 =
\left( \begin{array}{ccccc}
20 & -80 &   0 &   0 &   0 \\
20 &  20 & -80 &   0 &   0 \\
20 &  20 &  20 & -80 &   0 \\
20 &  20 &  20 &  20 & -80
\end{array} \right).
\end{align*}
\end{cor}

In \eqref{h5}, the asymptotic expansions of {\boldmath $K$}$(h)$ and {\boldmath $J$}$(h)$ near $h=0$ 
can obtained by substituting the asymptotic expansions of $x_j(h)$ and $y_j(h)$, $j=1,2$ into 
$$K_{i,1}(h) = x_2^i(h)y_2(h)-x_1^i(h)y_1(h), 
\quad J_i(h) = x_2^i(h) y_2(h)x'_2(h) - x_1^i(h) y_1(h)x'_1(h).$$
Then the asymptotic expansion of $\widetilde{\mbox{\boldmath $X$}}(h)$ can be derived by \eqref{h5}
with undetermined coefficients.
Note that $\widetilde{\mbox{\boldmath $X$}}(h)$ is analytic in $h$ for system \eqref{sys} with $0<\theta<\pi$.
For $\theta=\frac{\pi}{2}$, using
\[I^+_{i,1}(0)=2\!\int_1^{\frac{5}{4}}x^{i+2} \Big(\frac{1}{2}-\frac{2}{5}x\Big)^{\frac{1}{2}}\mathrm{d}x,\quad
\frac{\mathrm{d}}{\mathrm{d}h}I_{i,1}^+(0)=2\!\int_1^{\frac{5}{4}}x^{i-2} \Big(\frac{1}{2}-\frac{2}{5}x\Big)^{-\frac{1}{2}}\mathrm{d}x,\]
from \eqref{h5} we can get the asymptotic expansions of $I^+_{i,1}(h)$, $i=0,1,2,3$, as follows
\begin{equation}\label{zz1}
\begin{split}
I^+_{0,1}(h) = \,& \frac{17\sqrt{10}}{420} + \frac{2\sqrt{10}}{5}(1 + 4k_4)h - \frac{9\sqrt{10}}{3125}(13 + 1232k_4)h^2+\cdots,\\
I^+_{1,1}(h) = \,& \frac{113\sqrt{10}}{2520} + 4\sqrt{10} k_4h + \frac{\sqrt{10}}{250}(103 - 1008k_4)h^2+\cdots,\\
I^+_{2,1}(h) = \,& \frac{691\sqrt{10}}{13860} + \sqrt{10}h + \frac{4\sqrt{10}}{75}(19 - 84k_4)h^2+\cdots,\\
I^+_{3,1}(h) = \,& \frac{20047\sqrt{10}}{360360} + \frac{7\sqrt{10}}{6}h + \frac{3\sqrt{10}}{5}(3 - 8k_4)h^2+\cdots,\\
\end{split}
\end{equation}
where $k_4=\frac{\sqrt{5}}{5}\mbox{\rm arcsinh}(\frac{1}{2})$.
For $\theta \in (0,\frac{\pi}{2})\cup (\frac{\pi}{2},\pi)$, the asymptotic expansions of integrals $I^+_{i,1}(h)$
are omitted here, because the expressions of coefficients are too long.

\section{Proof of Theorems \ref{thm1} and \ref{thm2}}\label{sec4}


In this section, we shall prove Theorems \ref{thm1} and \ref{thm2} by using Corollary \ref{lmm3}, Lemma \ref{lmm6} and Theorem \ref{th4}.

\begin{proof}[Proof of Theorem \ref{thm1}]
Melnikov function $M_1(h)$ in \eqref{g1} can be simplified into
\begin{equation}\label{b2b}
\begin{split}
M_1(h)&=\, \oint_{\Gamma_h} \left(Q_1(x,y) + \int \frac{\partial}{\partial x} P_1(x,y)\mathrm{d}y\right)\mathrm{d}x\\
&=\, \sum_{i=0}^3 B_{i1} I_{i,1}(h) + \sum_{i=4}^5 B_{i1}I_{i-4,3}(h),
\end{split}
\end{equation}
where
\begin{equation}\label{b14b}
\begin{split}
&B_{01} = a_{101}+b_{011},\quad B_{11} = 2a_{201}+b_{111},\quad B_{21} = 3a_{301}+b_{211}, \\
& B_{31} = 4a_{401}+b_{311},\quad B_{41} = \frac{1}{3}a_{121}+b_{031},\quad B_{51} = \frac{2}{3}a_{221}+b_{131}.
\end{split}
\end{equation}
By the statements (ii) and (iii) of Corollary \ref{lmm3}, we obtain
\begin{equation}\label{b14}
\begin{split}
I_{0,3}(h)&=\,\frac{30}{17}hI_{0,1}(h)+\frac{3}{34}I_{3,1}(h),\\
I_{1,3}(h)&=\, \frac{15}{323}hI_{0, 1}(h) + \frac{30}{19}hI_{1, 1}(h) + \frac{105}{1292}I_{3, 1}(h).
\end{split}
\end{equation}
Then substituting  \eqref{b14} into \eqref{b2b} yields
\begin{equation*}\label{b15}
\begin{split}
M_1(h)=\,&\Big(B_{01}+\frac{30}{17} B_{41} h + \frac{15}{323} B_{51} h\Big)\,I_{0,1}(h)+\Big(B_{11}+\frac{30}{19}B_{51}h\Big)\,I_{1,1}(h)\\
&+B_{21}I_{2,1}(h) + \Big(B_{31}+\frac{3}{34}B_{41}+\frac{105}{1292}B_{51}\Big)\,I_{3,1}(h).
\end{split}
\end{equation*}
By \eqref{b13a} for $0<-h\ll 1$ $M_1(h)$ can be expanded as
\begin{equation}\label{b15a}
M_1(h)= c_{01} +  c_{11}|h|^{\frac{3}{4}} + c_{21}h\ln\!|h| +  c_{31}h + c_{41}|h|^{\frac{5}{4}} + c_{51}h|h|^{\frac{3}{4}}+\cdots,
\end{equation}
where
\begin{equation}\label{b16}
\begin{split}
c_{01} &=\,\sqrt{2}\Big(\frac{25}{84} B_{01}\!+\!\frac{125}{504} B_{11}\!+\!\frac{625}{2772} B_{21}\!+\!
\frac{15625}{72072} B_{31}\!+\!\frac{15625}{816816} B_{41}\!+\!\frac{78125}{4434144} B_{51}\Big), \\
c_{11} &=\,k_1B_{01},\quad
c_{21}= -\frac{\sqrt{2}}{10}(2 B_{01}+5 B_{11}),\\
c_{31} &=\,\frac{k_2}{2}(2B_{01} +5 B_{11}) +\sqrt{2}\Big(3 B_{11} +5 B_{21} +\frac{25}{6} B_{31} +\frac{25}{28} B_{41} +\frac{125}{168} B_{51}\Big), \\
c_{41} &=\,
\frac{k_3}{21}(21 B_{01}+30 B_{11}+50 B_{21}),\\
c_{51} &=\, 
- \frac{k_1}{1750}(663 B_{01}+780 B_{11}+900 B_{21}+1000 B_{31}-3000 B_{41}).
\end{split}
\end{equation}
By \eqref{b16}, it is straightforward to get
\begin{equation*}
c_{j1}=0,\;j=0,1,\ldots,5\ \Longleftrightarrow\ B_{j1}=0,\;j=0,1,\ldots,5.
\end{equation*}
Then by \eqref{b2b} and \eqref{b15a} $M_1(h)\equiv0$ if and only if $B_{01}=B_{11}=\cdots=B_{51}=0$, which yields \eqref{g2}. 

 Because $c_{01}$, $c_{11}$, $\ldots$, $c_{51}$ are linear in 
parameters $a_{ij1}$ and $b_{ij1}$, we have
\begin{align*}
{\rm rank}\, \frac{\partial(c_{01},\,c_{11},\,c_{21},\,c_{31},\,c_{41},\,c_{51})}{\partial (a_{101},\,a_{121},\,a_{201},\,a_{221},\,a_{301},\,b_{311})} =6.
\end{align*}
By Theorem \ref{th4}, $M_1(h)$ can have at most $5$ zeros (counting multiplicity) for $0<-h\ll 1$ for system \eqref{sys1},
and this upper bound can be reached for simple zeros with proper values of parameters. The proof is completed.
\end{proof}


\begin{proof}[Proof of Theorem \ref{thm2}]
By Theorem \ref{thm1} $P_1(x,y)$ and $Q_1(x,y)$ should satisfy \eqref{g2} if $M_1(h)\equiv0$. Then it is easy to get
\begin{equation*} 
\begin{split}
r_1(x,y)=\, -A_1x - A_2x^2 - \frac{1}{3}A_3x^3-\frac{1}{15}A_4x(5x^5-6x^4+15y^2),
\end{split}
\end{equation*}
from
\begin{equation*}
Q_1(x,y)\mathrm{d}x-P_1(x,y)\mathrm{d}y=r_1(x,y)\mathrm{d}H+\mathrm{d}R_1(x,y),
\end{equation*}
where $R_1(x,y)$ is a polynomial of degree 11 in $(x,y)$ and
\begin{equation} \label{g2b}
\begin{split}
&A_1 = a_{111} + 2b_{021}, \quad A_2 = a_{211} + b_{121},\\
&A_3 = 3a_{311} + 2b_{221},\quad A_4 = a_{131} + 4b_{041}.
\end{split}
\end{equation}
By using the Fran\c coise's algorithm \cite{F1996,I1998}, we have the second-order Melnikov function
\begin{equation}\label{M2A}
M_2(h)=\, \oint_{\Gamma_h} Q(x,y)\mathrm{d}x-P(x,y)\mathrm{d}y,
\end{equation}
where
\begin{equation*}
\begin{split}
Q(x,y)&=\, Q_2(x,y)+r_1(x,y)Q_1(x,y),\\
P(x,y)&=\, P_2(x,y)+r_1(x,y)P_1(x,y).
\end{split}
\end{equation*}
Then we can rewrite $M_2(h)$ in \eqref{M2A} to the form
\begin{equation}\label{M2}
\begin{split}
&M_2(h)=\, \oint_{\Gamma_h} \left(Q(x,y) + \int \frac{\partial}{\partial x} P(x,y)\mathrm{d}y\right)\mathrm{d}x\\
=&\, \sum_{i=0}^9 B_{i2} I_{i,1}(h) + \sum_{i=10}^{17} B_{i2}I_{i-10,3}(h) + \sum_{i=18}^{23} B_{i2}I_{i-18,5}(h) + B_{242}I_{0,7}(h),
\end{split}
\end{equation}
where coefficients $B_{i2}$, $i=0,1,\ldots,24$, are polynomials in parameters $a_{ijl}$ and $b_{ijl}$, $l=1,2$.
By Corollary \ref{lmm3} and \eqref{M2}, we can get
\begin{equation}\label{xx4}
\begin{split}
M_2(h)=\,& (B_{02}+C_1h+C_2h^2+C_3h^3)\,I_{0,1}(h)+(B_{12} + C_4h+C_5h^2)\,I_{1,1}(h)\\
& + (B_{22}+C_6h+C_7h^2)\,I_{2,1}(h) + (C_8+C_9h+C_{10}h^2)\,I_{3,1}(h).
\end{split}
\end{equation}
By \eqref{b13a} for $0<-h\ll 1$ $M_2(h)$ can be expanded as
\begin{equation*}
M_2(h)=\,c_{02} + c_{12}|h|^{\frac{3}{4}} + c_{22}h\ln\!|h|+c_{32}h + c_{42}|h|^{\frac{5}{4}} + c_{52}h|h|^{\frac{3}{4}} + \cdots,
\end{equation*}
where
\begin{align*}
c_{02}=\,& \sqrt{2}\Big(\frac{25}{84}B_{02} + \frac{125}{504}B_{12} + \frac{625}{2772}B_{22} + \frac{15625}{72072}B_{32} + \frac{15625}{72072}B_{42} \\
& +\! \frac{78125}{350064}B_{52} + \frac{390625}{1662804}B_{62} + \frac{1953125}{7759752}B_{72} + \frac{48828125}{178474296}B_{82} \\
& +\! \frac{9765625}{32449872}B_{92} + \frac{15625}{816816}B_{102}+ \frac{78125}{4434144}B_{112} + \frac{390625}{23279256}B_{122} \\
& +\! \frac{1953125}{118982864}B_{132} + \frac{1953125}{118982864}B_{142}+ \frac{9765625}{584097696}B_{152} \\
& +\! \frac{48828125}{2823138864}B_{162} + \frac{244140625}{13464200736}B_{172} + \frac{9765625}{6425074656}B_{182}\\
& +\! \frac{48828125}{33877666368}B_{192} + \frac{244140625}{175034609568}B_{202} + \frac{1708984375}{9394165199232}B_{212}\Big),\\
c_{12}=\,& k_1B_{02},\quad c_{22} =\, -\frac{\sqrt {2}}{10} ( 2B_{02} + 5B_{12}),\\
\cdots \,\,\,& \cdots.
\end{align*}

Note that coefficients $c_{i2}$, $i\ge 0$, are linear in parameters $a_{ij2}$ and $b_{ij2}$.
It is straightforward to get a unique solution in $a_{102},\,a_{122},\,b_{112},\,b_{212},\,b_{132}$ and $b_{312}$ from
\begin{equation}\label{q1}
\begin{split}
c_{02}=c_{12}=c_{22}=c_{32}=c_{42}=c_{52}=0.
\end{split}
\end{equation}
To finish the proof, we only need to discuss the solutions of
\begin{equation}\label{q2}
\begin{split}
c_{62}=c_{72}=c_{82}=c_{92}=c_{102}=c_{112}=0,
\end{split}
\end{equation}
with \eqref{q1} holding.

We can find all the solutions of \eqref{q2} by the following cases
\begin{equation*}
\begin{split}
&\,1. \,\,A_4\neq0;\\
&\,2. \,\,A_4=0,\quad A_3\neq0;\\
&\,3. \,\,A_4=0,\quad A_3=0,\quad A_2\neq0;\\
&\,4. \,\,A_4=0,\quad A_3=0,\quad A_2=0,\quad A_1\neq0;\\
&\,5. \,\,A_4=0,\quad A_3=0,\quad A_2=0,\quad A_1=0.\\
\end{split}
\end{equation*}
For case 1, because all $c_{i2}$, $i=6,7,\ldots,11$, are linear in $a_{201}$, $a_{221}$, $a_{401}$, $b_{031}$, $a_{021}$ and $a_{041}$, and
\begin{align*}
{\rm det}\,\frac{\partial(c_{62},\,c_{72},\,c_{82},\,c_{92},\,c_{102},\,c_{112})} {\partial (a_{201},\,a_{221},\,a_{401},\,b_{031},\,a_{021},\,a_{041})} 
= -\frac{325679779052734375000 k_1k_3A_4^6}{38041092419836172033367},
\end{align*}
we can get a unique solution in $a_{201},\,a_{221},\,a_{401},\,b_{031}$,\, $a_{021}$ and $a_{041}$ from \eqref{q2},
which implies $M_2(h)\equiv0$ by \eqref{xx4}. Note that by \eqref{g2b} $A_k$, $k=1,2,3,4$, are independent with 
${\mbox{\boldmath $\delta$}_{1}}=(a_{102},\,a_{122},\,b_{112},\,b_{212},\,b_{132},\,b_{312},\,
a_{201},\,a_{221},\,a_{401},\,b_{031},\,a_{021},\,a_{041})$.
Then we have
\begin{align*}
{\rm rank}\,\frac{\partial(c_{02},\,c_{12},\,\ldots,\,c_{112})}{\partial {\mbox{\boldmath $\delta$}_{1}}} =12\,\,\mbox{for}\,\,A_4\neq0.
\end{align*}
By Theorem \ref{th4}, $M_2(h)$ has at most $11$ zeros (counting multiplicity) for $0<-h\ll 1$ for $A_4 \neq0$,
and this upper bound can be reached.

The remaining cases can be similarly discussed. For the cases 2, 3, 4, no more than 10 simple zeros are found for $M_2(h)$ for $0<-h\ll 1$.
In the case 5, $M_2(h)$ is equivalent to zero with \eqref{q1} holding. The details are omitted here.
\end{proof}

\section{Proof of Theorem \ref{th3}}\label{sec5}
In this section, we shall present the proof of Theorem \ref{th3} 
for the cases of $\theta=\pi$, $\theta=\frac{\pi}{2}$ and $0<|\theta-\frac{\pi}{2}|<\frac{\pi}{2}$,
respectively.

\subsection{Case of $\theta = \pi$}\label{SubSec43}

In this case we have $\Gamma^\pm_h=\Gamma_h|_{\pm y\geq0}$.
Because of the symmetry of $\Gamma_h$, it is easy to get
\begin{equation}\label{a1}
\begin{split}
&I_{i,j}^{+}(h)=\frac{1}{2}I_{i,j}(h)\,\,\mbox{for}\,\, j\,\, \mbox{odd}.
\end{split}
\end{equation}
Since the separation line is $y=0$, we have $K_{i,j}(h)\equiv0$ for $j\geq1$.
Note that $4K_{5,0}(h)=5K_{4,0}(h)$ by \eqref{q3}.
Then by Corollary \ref{lmm3} and \eqref{a1}, for $\widetilde M(h)$ in \eqref{m1a} we have
\begin{equation}\label{a3}
\begin{split}
\widetilde{M}(h)=\,& (B_0 \!+\! B_1h)I_{0,1}(h) + (B_2\! +\! B_3h)I_{1,1}(h) + B_4I_{2,1}(h) + B_5I_{3,1}(h)
  \\& + (B_6+B_7h+48B_8h^2)K_{1,0}(h) + (B_9+B_{10}h)K_{2,0}(h)
  \\&+ (B_{11}+B_{12}h)K_{3,0}(h) + (B_{13}+B_8h)K_{4,0}(h),\\
\end{split}
\end{equation}
where $B_l$, $l=0,1,\ldots,13$, are linear combinations of free coefficients $a_{ij}^\pm$ and $b_{ij}^\pm$.
It is straightforward to verify that we can take $\mbox{\boldmath $\delta$}_{2}=(B_0,\,B_1,\,\ldots,\,B_{13})$ as a free parameter vector.

For \eqref{a3} in order to get the asymptotic expansion of $\widetilde{M}(h)$ for $0<-h\ll 1$,
we only need to investigate $K_{i,0}(h)$, $1\le i\le 4$. 
Because $\Gamma_h^+$ is the orbital arc from $(x_1(h),0)$ to $(x_2(h),0)$, we have
\begin{equation}\label{st5}
\begin{split}
&K_{i,0}(h)=\,x^{i}_2(h)-x^i_1(h).\\
\end{split}
\end{equation}
Then substituting \eqref{a6} and \eqref{a7} into \eqref{st5} we get the asymptotic expansions of $K_{i,0}(h)$, $1\le i\le 4$.

Then we can have the asymptotic expansion of $\widetilde{M}(h)$ for $0<-h\ll 1$ as follows
\begin{equation*}\label{m2}
\begin{split}
\widetilde{M}(h) =\,& c_0 + \sum_{j=0}^{+\infty}\Big(
c_{5j+1}h^j|h|^\frac{1}{4} +
c_{5j+2}h^j|h|^\frac{2}{4} +
c_{5j+3}h^j|h|^\frac{3}{4}\\
& \qquad\qquad   +
c_{5j+4}h^{j+1}\ln\!|h| +
c_{5j+5}h^{j+1}\Big),
\end{split}
\end{equation*}
where
\begin{align*}
&c_{0} =\, 25\Big( \frac{\sqrt{2}}{84}B_{0} + \frac{5\sqrt{2}}{504}B_{2} + \frac{25\sqrt{2}}{2772}B_{4} + \frac{625\sqrt{2}}{72072}B_{5} + \frac{1}{20}B_{6} + \frac{1}{16}B_{9}+ \frac{5}{64}B_{11}\\
&\qquad  + \frac{25}{256}B_{13} \Big),\quad c_{1} =\, -\sqrt{2}B_{6},\quad  c_{2} =\, - \frac{2}{5}B_{6} - 2B_{9},
\\
&
c_{3} =\, k_{1}B_{0} - \frac{\sqrt{2}}{25}(7B_{6} + 20B_{9} + 50B_{11}),
\\
&\cdots\,\, \cdots.
\end{align*}

Note that coefficients $c_i,\,i\geq0,$ are linear in $\mbox{\boldmath $\delta$}_{2}$.
Because
\begin{align*}
{\rm rank}\,\frac{\partial(c_0,\,c_1,\,\ldots,\,c_{12})}{\partial \mbox{\boldmath $\delta$}_{2}} =13,
\end{align*}
by \eqref{a3} we get $\widetilde M(h)\equiv0$ if and only if $c_0=c_1=\cdots=c_{12}=0$, which implies
$$c_j=O(|c_0,\,c_1,\,\ldots,\,c_{12}|),\quad j\ge 13.$$
By Theorem \ref{th4}, for $\theta=\pi$ there exists $0<\varepsilon_0\ll1$ such that $\widetilde M(h)$ has at most $12$ zeros (counting multiplicity)  for $h\in(-\varepsilon_0,0)$, 
and the upper bound can be reached by proper perturbations in system \eqref{sys}.


\subsection{Case of $\theta = \frac{\pi}{2}$}

In this case $\Gamma^\pm_h=\Gamma_h\bigcap \{\pm(x-1)\ge0\}$.
By Corollaries \ref{lmm3}, \ref{lm1} and \eqref{q4},  $\widetilde M(h)$ in \eqref{m1a} can be simplified as
\begin{equation}\label{xx5}
\begin{split}
\widetilde{M}(h)=&\, (\widetilde B_0 + \widetilde B_1h)I_{0,1}(h) + (\widetilde B_2 + \widetilde B_3h)I_{1,1}(h) + \widetilde B_4I_{2,1}(h)\\
& + \widetilde B_5I_{3,1}(h)+ (\widetilde B_6 + \widetilde B_7h)I^+_{0,1}(h) + (\widetilde B_8 + \widetilde B_9h)I^+_{1,1}(h) \\
&+ \widetilde B_{10}I^+_{2,1}(h)
 + \widetilde B_{11}I^+_{3,1}(h)+  (\widetilde B_{12}+\widetilde B_{13}h+\widetilde B_{14}h^2)K_{0,1}(h),
\end{split}
\end{equation}
where the coefficients $\widetilde B_{0},\,\widetilde B_{1},\,\ldots,\,\widetilde B_{14}$ can be treated as free parameters.
To obtain the asymptotic expansion of $\widetilde M(h)$ for $0<-h\ll 1$, we need to compute the corresponding asymptotic expansions
of $I_{0,1}^+(h)$, $I_{1,1}^+(h)$, $I_{2,1}^+(h)$, $I_{3,1}^+(h)$ and $K_{0,1}(h)$.

By \eqref{h12} we have
\begin{equation}\label{h15}
\begin{split}
K_{0,1}(h) &=\, y_2(h)-y_1(h)= -\frac{\sqrt{10 + 200h}}{5}\\
&=\, -\frac{\sqrt{10}}{5} - 2\sqrt{10}h + \cdots + 67031250\sqrt{10}h^8 + \cdots
\end{split}
\end{equation}
for $0<-h\ll1$. 
Then by Lemma \ref{lmm6}, \eqref{zz1} and \eqref{h15}  for $0<-h\ll 1$ $\widetilde{M}(h)$ in \eqref{xx5} can be expanded as
\begin{equation*}\label{}
\widetilde{M}(h) =\, \bar c_0 + \sum_{j=0}^{+\infty}\Big(\bar c_{4j+1}|h|^\frac{3}{4} + \bar c_{4j+2}h\ln\!|h| + \bar c_{4j+3}h + \bar c_{4j+4}|h|^\frac{5}{4}\Big)h^j,
\end{equation*}
where
\begin{equation*}\label{}
\begin{split}
\bar c_0 =\,& \sqrt{2} \Big(\frac{25}{84}\widetilde B_{0} + \frac{125}{504}\widetilde B_{2} + \frac{625}{2772}\widetilde B_{4} + \frac{1562}{72072}\widetilde B_{5} + \frac{17\sqrt{5}}{420}\widetilde B_{6}+ \frac{113\sqrt{5}}{2520}\widetilde B_{8}\\
&
+ \frac{691\sqrt{5}}{13860}\widetilde B_{10} + \frac{20047\sqrt{5}}{360360}\widetilde B_{11} - \frac{\sqrt{5}}{5}\widetilde B_{12}\Big),\\
\bar c_1 =\,& k_1 \widetilde B_0, \quad \widetilde c_2 = -\frac {\sqrt {2}}{10} ( 2\widetilde B_0+5\widetilde B_2),\\
\cdots\,\,\,&\cdots.
\end{split}
\end{equation*}

Note that coefficients $\bar c_i$, $i\ge0$, are linear in $\widetilde B_0$, $\widetilde B_1$, $\ldots$, $\widetilde B_{14}$.
It is straightforward to get $\bar c_{4j+1}=\bar c_{4j+2}=\bar c_{4j+4}=0$ for $j\ge2$ if $\bar c_0=\bar c_1=\cdots=\bar c_8=0$.
Let
\begin{equation*}
 \mbox{\boldmath $\bar c$} = (\bar c_0,\,\bar c_1,\,\ldots,\,\bar c_{8},\,\bar c_{11},\,\bar c_{15}
,\,\bar c_{19},\,\bar c_{23},\,\bar c_{27},\,\bar c_{31}),\,\,
 \mbox{\boldmath $\delta$}_{3} = (\widetilde B_0,\,\widetilde B_1,\,\ldots,\,\widetilde B_{14}).
\end{equation*}
We can get {\boldmath $\bar c $} $=0$ if and only if {\boldmath $\delta$}$_3=0$, i.e.
$$ \mbox{rank}\, \frac{\partial\mbox{\boldmath $\bar c$}}{\partial {\mbox{\boldmath $\delta$}_{3}}} =15.$$
Then by Theorem \ref{th4}, for $\theta=\frac{\pi}{2}$ there exists $0<\varepsilon_0\ll1$ such that $\widetilde M(h)$ has at most $14$  zeros (counting multiplicity)  for $h\in(-\varepsilon_0,0)$, 
and the upper bound can be reached by proper perturbations in system \eqref{sys}.


\subsection{Case of\, $0<|\theta-\frac{\pi}{2}|<\frac{\pi}{2}$} 


By Corollaries \ref{lmm3}, \ref{lm1} and Lemma \ref{lm3}, for $\widetilde M(h)$ in \eqref{m1a} we have
\begin{equation}\label{h21}
\begin{split}
\widetilde{M}(h)
=\,& (\widehat B_0 + \widehat B_1h)I_{0,1}(h) + (\widehat B_2 + \widehat B_3h)I_{1,1}(h) + \widehat B_4I_{2,1}(h) + \widehat B_5I_{3,1}(h)\\
&+ (\widehat B_6 + \widehat B_7h)I^+_{0,1}(h) + (\widehat B_8 + \widehat B_9h)I^+_{1,1}(h) + \widehat B_{10}I^+_{2,1}(h) + \widehat B_{11}I^+_{3,1}(h)\\
&+ (\widehat B_{12}+\widehat B_{13}h + 48k^3\widehat B_{14}h^2) K_{0,1}(h) + (\widehat B_{15}+\widehat B_{16}h)K_{0,2}(h)\\
&+ (\widehat B_{17}+\widehat B_{18}h)K_{0,3}(h) + (\widehat B_{19}+\widehat B_{14}h)K_{0,4}(h),
\end{split}
\end{equation}
where coefficients $\widehat B_i$, $0\le i\le 19$, can be taken as free parameters.

It is difficult to compute the asymptotic expansion of $\widetilde M(h)$ for $0<-h\ll 1$ if we keep taking $\theta$
as a free parameter.
For simplicity we take $\theta=\arctan\frac{121\sqrt{6}}{100}$. Then $x_a=\frac{11}{10}$,
$y_a=\frac{121}{1000}\sqrt{6}$ and
$$x_b=0.89791472161636914533040398\cdots$$ by \eqref{ha3}.
Using \eqref{h2} and \eqref{h3}, we get the power series expansions of {\boldmath $K$}$(h)$ and {\boldmath $J$}$(h)$ for $0<-h\ll 1$ for \eqref{h5},
by which we can further get the corresponding power series expansions of $I^+_{i,1}(h)$ near $h=0$, $i=0,1,2,3$.
By \eqref{h2} and \eqref{h3} we can also get the power series expansions of $K_{0,1}(h)$, $K_{0,2}(h)$, $K_{0,3}(h)$ and $K_{0,4}(h)$.
Then by Lemma \ref{lmm6} and \eqref{h21}, for $0<-h\ll 1$ we can compute the coefficients $\hat c_i$, $0\le i\le 51$ of $\widetilde{M}(h)$ given by
\begin{equation}\label{h24}
\widetilde{M}(h) =\, \hat c_0 + \sum_{j=0}^{+\infty}\Big(\hat c_{4j+1}|h|^\frac{3}{4} + \hat c_{4j+2}h\ln\!|h| + \hat c_{4j+3}h + \hat c_{4j+4}|h|^\frac{5}{4}\Big)h^j.
\end{equation}

In \eqref{h24} coefficients $\hat c_i,\,i\geq0,$ are linear in $\widehat B_j,\,j=0,1,\ldots,19$.
It is easy to get that $\hat c_1=\hat c_2=\hat c_4=\hat c_5=\hat c_6=\hat c_8=0$ if and only if $\widehat B_0=\widehat B_1=\widehat B_2=\widehat B_3=\widehat B_4=\widehat B_5=0$,
which further implies that all $\hat c_{4j+1}$, $\hat c_{4j+2}$ and $\hat c_{4j+4}$ vanish for $j\ge 2$ by \eqref{h21} and \eqref{h24}.
Because
\begin{equation*}
\begin{split}
{\rm det}\,\frac{\partial\mbox{\boldmath $\hat c$}}{\partial {\mbox{\boldmath $\delta$}_{4}}} \approx 1.9220607607207538701280448\times10^6\, k_1^2k_3^2\neq0,
\end{split}
\end{equation*}
where
$\mbox{\boldmath $\hat c$} =\, (\hat c_0,\,\hat c_1,\,\ldots,\,
\hat c_8,\,\hat c_{11},\,\hat c_{15},\,\hat c_{19},\,\ldots,\, \hat c_{51})$,
$\mbox{\boldmath $\delta$}_{4} =\, (\widehat B_0,\,\widehat B_1,\,\ldots,\,\widehat B_{19})$,
then by Theorem \ref{th4}, for $\theta=\arctan\frac{121\sqrt{6}}{100}$ there exists $0<\varepsilon_0\ll1$ such that $\widetilde M(h)$ has at most $19$ zeros (counting multiplicity)  for $h\in(-\varepsilon_0,0)$,
and the upper bound can be reached.
The proof is completed.

\section{Acknowledgments}
This work was supported by the National Natural Science Foundation of 
China (NSFC No. 12371175).


\end{document}